\documentclass[a4paper]{amsart}

\usepackage{amsmath}
\usepackage{cite}
\usepackage{dsfont}
\usepackage{color}
\usepackage{cancel}
\usepackage{amssymb}
\usepackage{bbm}
\usepackage[mathscr]{eucal}
\usepackage{mathtools}

\numberwithin{equation}{section}

\newtheorem{theorem}{Theorem}[section]
\newtheorem{definition}[theorem]{Definition}
\newtheorem{lemma}[theorem]{Lemma}
\newtheorem{remark}[theorem]{Remark}

\newtheorem{proposition}[theorem]{Proposition}

\newtheorem{example}[theorem]{Example}

\newcommand{\dist}{\operatorname{dist}}

\DeclarePairedDelimiter\ceil{\lceil}{\rceil}
\DeclarePairedDelimiter\floor{\lfloor}{\rfloor}
\begin{document}

\title{The Morse spectrum for linear random dynamical systems}
\author[R.~Al-Qaiwani]{Rayyan Al-Qaiwani}\address{Rayyan Al-Qaiwani and Martin Rasmussen, Department of Mathematics, Imperial College London, 180 Queen's Gate, London SW7 2AZ, United Kingdom}
\author[M.~Callaway]{Mark Callaway}
\author[M.~Rasmussen]{Martin Rasmussen}
\address{Mark Callaway, Centre for Environmental Mathematics, University of Exeter, Penryn, Cornwall, TR10 9FE, United Kingdom}

\date{\today}

\begin{abstract}

We prove that projectivised finite-dimensional linear random dynamical systems possess a unique finest weak Morse decomposition. Based on this result, we define the Morse spectrum and investigate its basic properties. In particular, we show that the Morse spectrum is given by a finite union of closed intervals. Furthermore we demonstrate that  under a bounded growth condition, the Morse spectrum coincides with the non-uniform dichotomy spectrum.

\end{abstract}

\maketitle

\section{Introduction}

In the context of linear random dynamical systems, spectral theory addresses both the characterisation of exponential growth behaviour and the decomposition into invariant subspaces to distinguish different growth. A major milestone was achieved in 1968 with Oseledets' Multiplicative Ergodic Theorem \cite{Oseledets1968}, which describes existence and properties of a spectrum of Lyapunov exponents. It states that, for a finite-dimensional linear random dynamical system  that satisfies a certain integrability condition and has ergodic base dynamics, almost surely, there exist finitely many Lyapunov exponents and a corresponding decomposition into invariant random linear spaces \cite[Chapter~3]{arnold1998random}. Recently, a different type of spectrum was developed for random dynamical systems: the dichotomy (or Sacker--Sell) spectrum \cite{Callaway_17_1}. This approach proved useful in particular to study bifurcations in random dynamical systems that cannot be captured by Lyapunov exponents. It originates from a series of papers on exponential dichotomies published by Sacker and Sell in the 1970's that culminated in \cite{SackerSell1978}, where the dichotomy spectrum was introduced for skew product flows with compact base.

Both the Lyapunov and the dichotomy spectrum lead to a spectral decomposition of the extended state space in the form of a Whitney sum. An alternative approach to study spectral properties of linear systems is to start with a natural decomposition in the form of a Whitney sum, and then to characterise the growth behaviour in the components of this decomposition.  This has first been studied for projectivised linear flows on vector bundles with a chain transitive base space. Selgrade's Theorem \cite{Selgrade_75_1,Salamon_88_1} establishes the existence of a finest Morse decomposition \cite{Conley_78_1} in this context. This enables one to associate to each component of the Morse decomposition a spectral interval of generalized
Lyapunov exponents, which yields the so-called Morse spectrum. This has been studied in the context of linear flows on vector bundles \cite{Colonius_96_1,Gruene_00_1} and nonautonomous dynamical systems \cite{Colonius_08_1,rasmussen2007alternative}.

In this paper, we develop the Morse spectrum for random dynamical systems. This requires as a first step to establish the existence of a unique finest Morse decomposition in the projective space for linear, finite-dimensional random dynamical systems. The focus in this paper is on weak Morse decompositions \cite{ochs1999weak,crauel2004towards}, and we note that Morse decompositions for pullback random attractors and repellers have been studied in \cite{Liu_08_1}. We demonstrate that the finest Morse decomposition corresponds to a Whitney sum of the extended state space. Building on this result, we define the Morse spectrum as the union of limit points of finite-time Lyapunov exponents of trajectories originating in the Morse sets of the unique finest Morse decomposition.
We then investigate  fundamental properties of the Morse spectrum, showing that it is almost surely constant and can be expressed as a finite union of closed intervals. Additionally, we analyse its behavior under random coordinate changes. Finally, we compare the Morse spectrum with the dichotomy spectrum for linear random dynamical systems. Specifically, we prove that under a bounded growth condition, the dichotomy spectrum coincides with the Morse spectrum, thus establishing the Morse spectrum as a constructive alternative to the dichotomy spectrum.

It is important to note that the Morse spectrum depends on the chosen notion of attractivity or repulsivity used to construct the unique finest Morse decomposition. In particular, it is influenced by whether attractors are permitted to approach their corresponding repellers arbitrarily closely. In the final section, we briefly address this dependency and introduce a uniform version of the Morse spectrum.

\section{Preliminaries}

In this section we introduce and review some preliminary notions that are needed in the main part of this paper. In particular, we provide a concise overview about the Morse decomposition theory for random dynamical systems and introduce the non-uniform dichotomy spectrum.

Let $(\Omega,\mathcal{F},\mu)$ be a complete probability space. A random dynamical system \cite{arnold1998random} on a Polish space $X$ is given by a pair of mappings $(\theta, \varphi)$: 
\begin{itemize}
    \item[(i)] a model of the randomness: an ergodic dynamical system $\theta:\mathbb{T}\times \Omega\to \Omega$, preserving the measure $\mu$. The dynamical system $\theta$ is called a \emph{metric dynamical system}.
    \item[(ii)] a model of the dynamics: a cocycle $\varphi: \mathbb{T}\times \Omega\times X\to X$, which is $(\mathcal{B}(\mathbb{T})\otimes \mathcal{F}\otimes \mathcal{B}(X),\mathcal{B}(X))$-measurable, and for $\omega\in \Omega$ and $x\in X$, one has:
    \begin{itemize}
        \item[$\bullet$] $\varphi(0,\omega)x=x$.
        \item[$\bullet$] for any $t,s\in \mathbb{T}$ one has $\varphi(t+s,\omega)=\varphi(t,\theta_{s}\omega)\circ \varphi(s,\omega)$.
        \item[$\bullet$] $(t,x)\mapsto \varphi(t,\omega)x$ is continuous.
    \end{itemize}
    Note that we write $\varphi(t,\omega,x)=\varphi(t,\omega)x$.
\end{itemize}

We will always assume that the time set $\mathbb{T}$ is either $\mathbb{T}=\mathbb{Z}$ or $\mathbb{T}=\mathbb{R}$.

We call a random dynamical system $(\theta,\varphi)$ with state space $\mathbb{R}^{d}$ linear, if for any $(t,\omega)\in \mathbb{T}\times \Omega$, the map $\varphi(t,\omega):\mathbb{R}^{d}\to \mathbb{R}^{d}$, $x\mapsto \varphi(t,\omega)x$ is linear. For such a linear random dynamical system $\varphi$ there exists a matrix-valued map $\Phi:\mathbb{T}\times \Omega \to \mathbb{R}^{d\times d}$ such that one has $\Phi(t,\omega)x=\varphi(t,\omega)x$ for any $(t,\omega,x)\in \mathbb{T}\times \Omega \times \mathbb{R}^{d}$. We identify $\varphi$ with $\Phi$. Unless stated otherwise, we equip $\mathbb{R}^{d}$ with the Euclidean norm $\|\cdot\|_{2}$.

We work on the real projective space $\mathbb{P}^{d-1}$. That is $\mathbb{P}^{d-1}=(\mathbb{R}^{d}\setminus \{0\})/\sim$ where for $(x,y)\in (\mathbb{R}^{d}\setminus \{0\})\times (\mathbb{R}^{d}\setminus \{0\})$ we define $x\sim y$ if and only if $x$ and $y$ are colinear. We denote the equivalence class of $x\in \mathbb{R}^{d}\setminus \{0\}$ by $\mathbb{P}x$. $\mathbb{P}^{d-1}$ is a compact metric space when equipped with the metric
$d_{\mathbb{P}}:\mathbb{P}^{d-1}\times \mathbb{P}^{d-1}\to [0,\sqrt{2}]$
\begin{align}
\label{metric on projected space}
d_{\mathbb{P}}(\hat{x},\hat{y})=\min \left\{\left|\frac{x}{\|x\|}-\frac{y}{\|y\|}\right|,\left|\frac{x}{\|x\|}+\frac{y}{\|y\|}\right|\right\}.
\end{align}
A linear two-sided random dynamical system $\Phi$  induces a projectivised linear random dynamical system $\mathbb{P}\Phi$ on $\mathbb{P}^{d-1}$ defined by 
\begin{align}
    \label{eq projected random dynamical system}
    \mathbb{P}\Phi(t,\omega)\mathbb{P}x:=\mathbb{P}(\Phi(t,\omega)x) \quad \text{ for all }(t,\omega,x)\in \mathbb{T}\times \Omega \times (\mathbb{R}^{d}\setminus \{0\}).
\end{align}
This set-up is discussed in a more general context in \cite[Section 6.2]{arnold1998random}.

Let $(Y,d)$ be a metric space. We define the Hausdorff semi-distance between two non-empty subsets $A,B\subset Y$ to be
\[
\dist : 2^Y \times 2^Y \to \mathbb{R}^+_0 \cup \{\infty\}, \ \dist(A,B) := \sup_{a \in A} \inf_{b \in B} d(a, b),
\]
We also set
\[
\dist(\emptyset, \emptyset) := 0, \quad \dist(\emptyset, A) := 0 \quad \text{and} \quad \dist(A, \emptyset) := \infty.
\]
Furthermore we define 
\[
\tilde{d}(A,B) := \inf_{a \in A, b \in B} d(a, b),
\]
with $\tilde{d}(A, \emptyset) = \tilde{d}(\emptyset, A) := 0$ and $\tilde{d}(\emptyset, \emptyset) := 0$.

In the context of random dynamical systems, one  often deals with $\omega$-indexed sets, so-called random sets \cite[Definition~14]{Cra2015RandAttr}. 
\begin{definition}[Random set]\label{def:random_set}
    Let $(\Omega, \mathcal{F}, \mu)$ be a complete probability space and $X$ a Polish space. A set $D\in \mathcal{B}(X)\otimes \mathcal{F}$ is called a random set. The $\omega$-fibers for some $\omega\in \Omega$ of a random set $D$ are defined by 
    $$D(\omega)=\{x\in X: (x,\omega)\in D\}.$$
    A random set $D$ is closed, open or compact if $D(\omega)$ is closed, open or compact respectively for any $\omega\in \Omega$.
\end{definition}
Next we require several notions of invariance for random sets $D$.
\begin{definition}[Invariance]
    Let $(\theta,\varphi)$ be a random dynamical system and consider a random set $D$.
    \begin{itemize}
        \item[(i)] $D$ is called \emph{forward invariant} if for all $\omega \in \Omega$ and $t \geq 0$
        \[
        \varphi(t, \omega) D(\omega) \subseteq D(\theta_t \omega).
        \]
   
        \item[(ii)] $D$ is called \emph{backward invariant} if for all $\omega \in \Omega$ and $t \leq 0$
        \[
        \varphi(t, \omega) D(\omega) \subseteq D(\theta_t \omega).
        \]
        
        \item[(iii)] $D$ is called \emph{invariant} if for all $\omega \in \Omega$ and $t \in \mathbb{T}$
        \[
        \varphi(t, \omega) D(\omega) = D(\theta_t \omega).
        \]
    \end{itemize}
\end{definition}
We  now present the main elements of Morse theory for weak attractors, following \cite{ochs1999weak,crauel2004towards}.

Generally, there are several reasonable notions of attractors for random dynamical systems. In this article, we focus on weak attractors \cite[Definition~4.1]{crauel2004towards}.

\begin{definition}[Local weak attractor]
An invariant compact random set $A$ is called a \emph{local weak attractor} if there exists a forward invariant open random set $U$ with $U(\omega) \supset A(\omega)$ $\mu$-a.s.~such that each closed random set $C \subset U$ is weakly attracted to $A$, i.e.,
\[
\lim_{t \to \infty} \mu \left\{ \omega : \dist(\varphi(t, \omega)C(\omega), A(\theta_t \omega)) > \varepsilon \right\} = 0 \quad \text{for every } \varepsilon > 0.
\]
In this case the neighborhood $U$ is said to be a weak attracting neighborhood of $A$. The set
\[
B(A)(\omega) := \left\{ x \in X : \varphi(t, \omega)x \in U(\theta_t \omega) \text{ for some } t \geq 0 \right\}
\]
is called the basin of attraction of $A$.
\end{definition}

A weak local repeller is defined via time-reversal \cite[Definition~4.1]{crauel2004towards}.
\begin{definition} [Local weak repeller]
An invariant compact random set $R$ is called a \emph{local weak repeller} if there exists a backward invariant open random set $U$ with $U(\omega) \supset R(\omega)$ $\mu$-a.s. such that each closed random set $C \subset U$ is weakly repelled by $R$, i.e.,
\[
\lim_{t \to -\infty} \mu \left\{ \omega : \dist(\varphi(t, \omega)C(\omega), R(\theta_t \omega)) > \varepsilon \right\} = 0 \quad \text{for every } \varepsilon > 0.
\]
The neighborhood $U$ is said to be a weak repelling neighborhood of $R$. The set
\[
B(R)(\omega) := \left\{ x \in X : \varphi(t, \omega)x \in U(\theta_t \omega) \text{ for some } t \leq 0 \right\}
\]
is called the basin of repulsion of $R$.
\end{definition}

In the following, we only present results concerning attractors, with analogous results holding for repellers. The first results, from \cite[Lemma 4.5, Corollary 5.2]{crauel2002random}, concern the basin of attraction.

\begin{lemma}[ ]
The basin of attraction of a local weak attractor $A$ with weak attracting neighborhood $U$ is an invariant open random set and $A$ weakly attracts all closed random sets $C$ such that $C(\omega) \subset B(A)(\omega)$ $\mu$-a.s., i.e.,
\[
\lim_{t \to \infty} \mu \left\{ \omega : \dist(\varphi(t, \omega)C(\omega), A(\theta_t \omega)) > \varepsilon \right\} = 0 \quad \text{for every } \varepsilon > 0.
\]
Moreover $B(A)$ is independent of $U$ up to a set of zero measure.
\end{lemma}

The following result from \cite[Proposition~5.1]{crauel2002random} tells us that, given a local weak attractor $\emptyset \neq A \subsetneq X$, one can construct a corresponding local weak repeller $\emptyset \neq R \subsetneq X$.   

\begin{proposition}
Let $A$ be a local weak attractor. Then $R := X \setminus B(A)$ is a local weak repeller with basin of repulsion $B(R) = X \setminus A$.
\end{proposition}

This leads to the notion of a weak attractor-repeller pair. 
\begin{definition}[Attractor-repeller pair]
Let $A$ be a local weak attractor, then the local weak repeller $R = X \setminus B(A)$ is called the repeller corresponding to $A$, and $A$ the attractor corresponding to $R$, and $(A,R)$ is called a \emph{weak attractor-repeller pair}.
\end{definition}

The following proposition \cite[Theorem~5.1]{crauel2004towards} is fundamental in order to obtain Morse decompositions constructed from weak attractor-repeller pairs. It states that given a nested finite sequence of attractors, the sequence of corresponding repellers is nested too.

\begin{proposition}
 Suppose that $A_1$ and $A_2$ are local weak attractors such that $A_1(\omega) \subsetneq A_2(\omega)$ $\mu$-a.s., and with corresponding repellers $R_1$ and $R_2$, then $R_1(\omega) \supsetneq R_2(\omega)$ $\mu$-a.s.
\end{proposition}

The notion of a weak Morse decomposition is given as follows 
\begin{definition}[Weak Morse decomposition]
Suppose that $(A_i, R_i)$, $i \in \{0, \dots, n\}$, are weak attractor-repeller pairs that almost surely satisfy
\[
\emptyset = A_0(\omega) \subsetneq A_1(\omega) \subsetneq \dots \subsetneq A_n(\omega) = X.
\]
Then the set $\mathcal{M} := \{M_1, \dots, M_n\}$, defined by
\[
M_i := A_i \cap R_{i-1}, \quad i \in \{1, \dots, n\}
\]
is called a \emph{weak Morse decomposition} of $X$, and the sets $M_i$ are called \emph{Morse sets}.
\end{definition}

We are interested in \emph{finest} weak Morse decompositions. A Morse decomposition $\mathcal{M} = \{M_1, \dots, M_n\}$ is said to be finer than the Morse decomposition $\tilde{\mathcal{M}} = \{\tilde{M}_1, \dots, \tilde{M}_m\}$ if for each $i \in \{1, \dots, m\}$, there exists $j \in \{1, \dots, n\}$ such that $M_j(\omega) \subset \tilde{M}_i(\omega)$ $\mu$-a.s. A minimal element of this partial ordering is called a finest Morse decomposition \cite[Definition 2.3.12]{Call2014}.

We now introduce a non-uniform version of the  dichotomy spectrum for random dynamical systems, which is a variation from \cite{Callaway_17_1}. Broadly, a linear random dynamical system is said to admit an exponential dichotomy if $\mathbb{R}^{d}$ can be split up into two random complementary invariant sub-spaces $U(\omega)$ and $V(\omega)$ of initial values such that the forward orbits of $\Phi$ with initial values in $U(\omega)$ (respectively $V(\omega)$) are exponentially bounded in forwards time (respectively backwards time) by distinct exponential growth rates. In \cite{Callaway_17_1}, it was (implicitly) demanded that the subspaces $U(\omega)$ and $V(\omega)$ are essentially uniformly bounded away from each other in the projective space. In our setting, however, we allow the angle between $U(\omega)$ and $V(\omega)$ to come arbitrarily close to zero under variation of $\omega$ in a controlled manner. The role of $U$ and $V$ is formalised by so called \emph{invariant projectors}:
\begin{definition}[Invariant projector]
    A measurable map $P:\Omega \to \mathbb{R}^{d\times d}$ is said to be an invariant projector for a linear random dynamical system $\Phi$ if the following two conditions are fulfilled:
    \begin{enumerate}
        \item $P(\omega)^{2}=P(\omega)$ for all $\omega\in \Omega$.
        \item   $ P(\theta_{t}\omega)\Phi(t,\omega)=\Phi(t,\omega)P(\omega) \text{ for all } t\in \mathbb{T}  \text{ and for all }  \omega\in \Omega.$
    \end{enumerate}

\end{definition}

The range $\mathcal{R}(P)$ and null-space $\mathcal{N}(P)$ of an invariant projector $P$ are invariant, linear random sets, such that the dimensions are almost surely constant \cite[Proposition~2.1]{Callaway_17_1}.

The following definition of an exponential dichotomy is slightly more general than the definition given in \cite{Callaway_17_1} and makes use of tempered random variables. A positive random variable $x$ is called \emph{tempered} if $\lim_{t\to\pm\infty}\frac{1}{t}\ln x(\theta_{t}\omega)=0$ almost surely (see \cite[Section 4.1.1]{arnold1998random}).
\begin{definition}[Non-uniform exponential dichotomy]
    \label{def exponential dichotomy}
    A linear random dynamical system $\Phi$ is said to admit a \emph{non-uniform exponential dichotomy} with growth-rate $\gamma\in \mathbb{R}$ if there exists an invariant projector $P_{\gamma}$ for $\Phi$, a constant $\alpha>0$ and a tempered random variable $K:\Omega \to [1,\infty)$ so that almost surely 
    \begin{align}
    \label{def exp dichotomy pos time}
    \|\Phi(t,\omega)P_{\gamma}(\omega)\|\leq K(\omega)e^{(\gamma-\alpha)t} \text{ for any }t\geq0
    \end{align}
    
    and 
    \begin{align}
    \label{def exp dichotomy neg time}
    \|\Phi(t,\omega)(\mathbbm{1}-P_{\gamma})(\omega)\|\leq K(\omega)e^{(\gamma+\alpha)t} \text{ for any }t\leq0.
    \end{align}
\end{definition}

If the random variable $K$ in the definition above is essentially bounded,  then the range of the invariant projector is essentially uniformly bounded away from the null-space of the invariant projector in the projective space \cite[Lemma 3.28]{alqaiwani2024spectra}.

\begin{definition}[Non-uniform dichotomy spectrum]
  The \emph{non-uniform dichotomy spectrum} is defined by  
  $$\Sigma^{\prime} =\{\gamma\in\overline{\mathbb{R}}:\Phi \text{ does not admit an  exponential dichotomy with growth rate }\gamma\}.$$
\end{definition}

The following result \cite{alqaiwani2024spectra} says that invariant projectors give rise to weak attractor-repeller pairs.
\begin{theorem}
\label{structure invariant manifolds}
    Let $\Phi$ be a linear random dynamical system that admits an exponential dichotomy with growth-rate $\gamma\in\overline{\mathbb{R}}$ and invariant projector $P_{\gamma}$. Then $(A,R):=(\mathbb{P}\mathcal{N}(P_{\gamma}),\mathbb{P}\mathcal{R}(P_{\gamma}))$ defines a weak attractor-repeller pair.
\end{theorem}

\section{The Morse spectrum}

We first prove that the projectivised linear random dynamical system $\mathbb{P}\Phi$ possesses a unique finest weak Morse decomposition. 

In the following we assume that whenever $\mathbb{T}=\mathbb{R}$, the probability space $\Omega$ is also a metric space and $t\mapsto \theta_{t}(\omega)$ is continuous for any $\omega\in \Omega$. 

The following lemma is taken from \cite[Lemma A1]{kloeden2011nonautonomous}.
\begin{lemma}
\label{metric projective space lemma}
    For all $\eta>0$, there exists a $\delta\in (0,1)$ such that for any $x,y\in \mathbb{R}^{d}\setminus \{0\}$ with 
    $$\frac{\langle x,y \rangle ^{2}}{\|x\|^{2}\|y\|^{2}}\geq 1-\delta$$
    one has $$d_{\mathbb{P}}(\mathbb{P}x,\mathbb{P}y)\leq \eta.$$
\end{lemma}
\begin{lemma}
\label{Lemma 1 proof weak attractor corresponds to subspace}
Let $(\Omega,\mathcal{F},\mu,(\theta_{t})_{t\in\mathbb{T}})$ be an invertible measure preserving dynamical system. Suppose $X:\Omega\to\mathbb{R}_{>0}$ is a positive random variable, then the family of random variables $(X_{t}(\cdot)=X(\theta_{-t}\cdot))_{t\in\mathbb{T}}$ does not converge to zero in probability for $(t\to \infty)$. As a consequence there exists a sequence $(t_{k})_{k\in\mathbb{N}}$, $t_{k}\uparrow \infty$ such that for any subsequence $(t_{k_{m}})_{m\in\mathbb{N}}$ of $(t_{k})_{k\in\mathbb{N}}$ there exists a positive measure set $A\in \mathcal{F}$ with 
$$\limsup_{m\to\infty} X(\theta_{-t_{k_{m}}}\omega)>0 \text{ for any } \omega\in A.$$
\end{lemma}

\begin{proof}
That $(X_{t})_{t\in \mathbb{T}}$ does not converge to zero in probability follows from the fact that any $X_{t}$, $t\in\mathbb{T}$ has the same distribution as $X$ due to measure preservation. The second part of the statement follows from the fact that $\lim_{t\to\infty}X_{t}=0$ (in probability) if and only if any sequence $(t_{k})_{k\in\mathbb{N}}$ with $t_{k}\uparrow \infty$ admits a subsequence $(t_{k_{m}})_{m\in\mathbb{N}}$ with the property that $\lim_{m\to\infty} X_{t_{k_{m}}}=0$ almost surely.
\end{proof}
\begin{lemma}
\label{Lemma 2 proof weak attractor corresponds to subspace}
Suppose that $\mathbb{P}^{-1}A(\omega)$ is an invariant linear subspace of $\mathbb{R}^{d}$ for all $\omega\in B$ with $B\in \mathcal{F}$ having positive measure. Then $\mathbb{P}^{-1}A$ is an invariant linear subspace of $\mathbb{R}^{d}$ almost surely.
\end{lemma}
\begin{proof}
Let $\tilde{\Omega}:=\bigcup_{t\in \mathbb{T}}\theta_{t}(B)$, then by ergodicity $\tilde{\Omega}$ has full measure. Please note that $\tilde{\Omega}:=\bigcup_{t\in \mathbb{T}}\theta_{t}(B)$ is measurable as $t\mapsto \theta_{t}\omega$ is continuous for every $\omega\in \Omega$ which implies that $\tilde{\Omega}:=\overline{\bigcup_{t\in \mathbb{T}\cap\mathbb{Q}}\theta_{t}(B)}$.  Let $\tilde{\omega}\in \tilde{\Omega}$, then $\tilde{\omega}=\theta_{t}(\omega)$ for some $t\in\mathbb{T}$ and some $\omega\in B$. Hence by invariance of $A$ one has $\mathbb{P}^{-1}A(\tilde{\omega})=\mathbb{P}^{-1}A(\theta_{t}\omega)=\Phi(t,\omega)\mathbb{P}^{-1}A(\omega)$. Now as $ \Phi(t,\omega)$ is a linear mapping the statement follows.
\end{proof}

The following proposition says that, given a weak attractor $A$ of $\mathbb{P}\Phi$, $\mathbb{P}^{-1}A(\omega)$ is a linear subspace of $\mathbb{R}^{d}$ for almost every $\omega\in \Omega$. This result is crucial for the existence of a unique finest weak Morse decomposition since it implies that there can be at most $d+1$ nested weak attractors. The proof is an adaption of the proof given in \cite[Proposition~3.3.3]{Call2014} which demonstrates that a uniform pullback attractor of $\mathbb{P}\Phi$ corresponds to a subspace in $\mathbb{R}^{d}$ almost surely. 
\begin{proposition}
\label{attractor corresponds to subspace}
Let $A$ be a local weak attractor for $\mathbb{P}\Phi$, then $\mathbb{P}^{-1}A$ is almost surely a linear subspace of $\mathbb{R}^{d}$.
\end{proposition}
\begin{proof}   
For fixed $\omega\in \Omega$ let $C\subset \mathbb{S}^{d-1}\setminus \mathbb{P}^{-1}A(\omega)$ denote an arbitrary compact set. For $u,v\in \mathbb{R}^{d}$, write $L_{u,v}:=\text{span}\{u,v\}$. The proof is divided into six steps.
\\
\emph{Step 1:}
It will be shown that there exists a positive measure set $B\in \mathcal{F}$ such that for any $\omega\in B$ there exists a sequence $t_{k}\uparrow \infty$ and a number $\alpha>0$ with 
$$
\lim_{k\to\infty}\mathrm{dist}_{\mathbb{P}}(\mathbb{P}\Phi(t_{k},\theta_{-t_{k}}\omega)\overline{B}_{\alpha}(A(\theta_{-t_{k}}\omega)),A(\omega))=0.
$$

Let $U$ be a fundamental neighborhood of $A$ and consider the following random variable:
$$\alpha:\Omega \to \mathbb{R}_{>0}, \ \omega\mapsto \sup\{\varepsilon>0:\overline{B}_{2\varepsilon}(A(\omega))\subset U(\omega)\}.$$
Then 
\begin{align}
\label{Convergence in prob}
\lim_{t\to\infty}\mathrm{dist}_{\mathbb{P}}(\mathbb{P}\Phi(t,\theta_{-t}\omega)\overline{B}_{\alpha(\theta_{-t}\omega)}(A(\theta_{-t}\omega)),A(\omega))=0 \text{ in probability. }
\end{align}
By Lemma \ref{Lemma 1 proof weak attractor corresponds to subspace} there exists a sequence $t_{k}\uparrow\infty$ such that for any subsequence $t_{k_{m}}\uparrow\infty$ there exists $B\in \mathcal{F}$ with positive measure such that 
\begin{align}
\label{limsup}
\limsup_{m\to\infty} \alpha(\theta_{-t_{k_{m}}}\omega)>0 \text{ for any } \omega\in B.
\end{align}
Let $(t_{k_{m}})_{m\in\mathbb{N}}$ be a subsequence of $(t_{k})_{k\in\mathbb{N}}$ such that \eqref{Convergence in prob} holds for any $\omega\in F$ along $(t_{k_{m}})_{m\in\mathbb{N}}$ where $F\in \mathcal{F}$ is a full measure set.  Let $B\in \mathcal{F}$ be the positive measure set such that \eqref{limsup} holds for $(t_{k_{m}})_{m\in\mathbb{N}}$ and let $\tilde{B}=F\cap B$. Now for any $\omega \in \tilde{B}$ we have that $\eqref{Convergence in prob}$ and \eqref{limsup} hold. Now let $\omega\in \tilde{B}$, by passing to a further subsequence along which the $\limsup$ in \eqref{limsup} is obtained  which we label w.l.o.g again by $(t_{k})_{k\in\mathbb{N}}$ we have that 
$$
\lim_{k\to\infty}\mathrm{dist}_{\mathbb{P}}(\mathbb{P}\Phi(t_{k},\theta_{-t_{k}}\omega)\overline{B}_{\alpha}(A(\theta_{-t_{k}}\omega)),A(\omega))=0.
$$
with $\alpha=\inf_{k\in\mathbb{N}}\alpha(\theta_{-t_{k}}\omega)>0$.
\\
\emph{Step 2:} Let $\tilde{B}\in \mathcal{F}$, $\alpha>0$ and $(t_{k})_{k\in\mathbb{N}}$ be as in Step 1. We show that for any $\omega\in \tilde{B}$, any \( u \in \mathbb{P}^{-1}A(\omega)\setminus \{0\} \) and \( v \in C \) such that \( \mathbb{P}u \) is a boundary point of \( A(\omega) \cap \mathbb{P}L_{u,v} \) with respect to \( \mathbb{P}L_{u,v} \), one has

\begin{equation}
\label{quotient limit proof subspace}
\lim_{k \to \infty} \frac{\|\Phi(-t_{k}, \omega)u\|}{\|\Phi(-t_{k}, \omega)v\|} = 0. 
\end{equation}

By Step 1,

\begin{equation}
\label{Application Step 1 Prop linear subspace}
\lim_{k \to \infty} \mathrm{dist}_\mathbb{P}(\mathbb{P}\Phi(t_{k}, \theta_{-t_{k}}\omega)B_\alpha(A(\theta_{-t_{k}}\omega)), A(\omega)) = 0 \quad \text{ for any }\omega\in \tilde{B}
\end{equation}

Now, by Lemma \ref{metric projective space lemma} there exists a \( \delta \in (0, 1) \) such that \( d_\mathbb{P}(\mathbb{P}x, \mathbb{P}y) \leq \frac{\alpha}{2} \) holds for all \( x, y \in \mathbb{R}^d \setminus \{0\} \) with

\begin{equation}
\label{delta inequality prop lin subspace}
\frac{\langle x, y \rangle^2}{\|x\|^2 \|y\|^2} \geq 1 - \delta. 
\end{equation}

To the aim of a contradiction assume that there exists an \( \omega \in \tilde{B} \) such that \eqref{quotient limit proof subspace} does not hold. Then there exists a \( \gamma > 0 \) and a sub-sequence of $(t_{k})_{k\in\mathbb{N}}$ which we label w.l.o.g again by  $(t_{k})_{k\in\mathbb{N}}$ such that

\[
\frac{\|\Phi(-t_k, \omega)v\|}{\|\Phi(-t_k, \omega)u\|} \leq \gamma, \quad \text{for all } k \in \mathbb{N}.
\]

For $c\neq0$ with  \( |c| \) sufficiently small, we obtain for \( k \in \mathbb{N} \):

\[
\frac{\langle \Phi(-t_k, \omega)(cv + u), \Phi(-t_k, \omega)u \rangle^2}{\|\Phi(-t_k, \omega)(cv + u)\|^2 \|\Phi(-t_k, \omega)u\|^2} \geq 1 - \delta.
\]

Hence taking \( |c| \) sufficiently small, \eqref{delta inequality prop lin subspace} implies 

\[
\mathrm{dist}_\mathbb{P}(\mathbb{P}\Phi(-t_k, \omega)\mathbb{P}(cv + u), A(\theta_{-t_k}\omega)) \leq \frac{\alpha}{2}, \quad \text{for all } k \in \mathbb{N}.
\]

This yields:
\[
\mathrm{dist}_\mathbb{P}(\mathbb{P}(cv + u), A(\omega)) = \lim_{k \to \infty} \mathrm{dist}_\mathbb{P}(\mathbb{P}\Phi(t_k, \theta_{-t_k}\omega)\mathbb{P}\Phi(-t_k, \omega)\mathbb{P}(cv + u), A(\omega)) = 0,
\]

This is a contradiction to the assumption that \( \mathbb{P}u \) is a boundary point of \( A(\omega) \cap \mathbb{P}L_{u,v} \) in \( \mathbb{P}L_{u,v} \).
\\
\emph{Step 3: }
We show that for any \( \omega \in \tilde{B} \), any \( u \in \mathbb{P}^{-1}A(\omega)\setminus \{0\} \) and \( v \in C \), the intersection \( A(\omega) \cap \mathbb{P}L_{u,v} \) consists of one point.

Note that any point in \( \mathbb{P}L_{u,v} \setminus \{\mathbb{P}u\} \) can be expressed by \( \mathbb{P}(v + cu) \) for some \( c \in \mathbb{R} \). From Step 2 we obtain that for any \( \omega \in \tilde{B} \):

\[
\lim_{k \to \infty} \frac{\langle \Phi(-t_{k}, \omega)(v + cu), \Phi(-t_{k}, \omega)v \rangle^2}{\|\Phi(-t_{k}, \omega)(v + cu)\|^2 \|\Phi(-t_{k}, \omega)v\|^2} = 1
\]

whenever \( \mathbb{P}u \) is a boundary point of \( A(\omega) \cap \mathbb{P}L_{u,v} \) relative to \( \mathbb{P}L_{u,v} \). This yields

\[
\lim_{k \to \infty} d_\mathbb{P}(\mathbb{P}\Phi(-t_{k}, \omega)\mathbb{P}(v + cu), \mathbb{P}\Phi(-t_{k}, \omega)\mathbb{P}v) = 0. \quad 
\]

Suppose \( \mathbb{P}(v + cu) \in A(\omega) \), then there is a \( K \in \mathbb{N} \) such that \( \mathbb{P}\Phi(-t_{k}, \omega)\mathbb{P}v \in B_\alpha(A(\theta_{-t_{k}} \omega)) \) for all \( k \geq K \), and hence

\[
\mathrm{dist}_\mathbb{P}(\mathbb{P}v, A(\omega)) = \lim_{k \to \infty} \mathrm{dist}_\mathbb{P}(\mathbb{P}\Phi(t_{k}, \theta_{-t_{k}} \omega)\mathbb{P}\Phi(t, \omega)Pv, A(\omega)) = 0,
\]

which contradicts the fact that $\mathbb{P}C$ is compact and  $\mathbb{P}v \in\mathbb{P}C$. Hence, \( A(\omega) \cap \mathbb{P}L_{u,v} \) consists of a single point.
\\
\emph{Step 4:}
It now follows directly from Steps 2 and 3 that for any \( \omega \in \tilde{B} \), and all \( u \in \mathbb{P}^{-1}A(\omega)\setminus \{0\} \) and \( v \in C \), one has

\[
\lim_{k \to \infty} \frac{\|\Phi(-t_{k}, \omega)u\|}{\|\Phi(-t_{k}, \omega)v\|} = 0.
\]
\\
\emph{Step 5:} We are now ready to prove that \( \mathbb{P}^{-1}A(\omega) \) is a linear subspace of \( \mathbb{R}^d \) for $\omega\in \tilde{B}$. From step 3 we have that for \( \omega \in \tilde{B} \), \( u \in \mathbb{P}^{-1}A(\omega) \) and \( v \in C \), \( A(\omega) \cap \mathbb{P}L_{u,v} \) consists of a single point. This implies that for any  \( x, y \in \mathbb{R}^d\setminus \{0\} \), it holds that \( A(\omega) \cap \mathbb{P}L_{x,y} \) is either a single point, the empty set, or equal to \( \mathbb{P}L_{x,y} \). This implies the statement.
\\
\emph{Step 6:}
Using Lemma \ref{Lemma 2 proof weak attractor corresponds to subspace} we obtain that $\mathbb{P}^{-1}A(\omega)$ is a linear subspace of $\mathbb{R}^{d}$ almost surely. 
\end{proof}

The following proposition says that a weak attractor-repeller pair $(A,R)$ corresponds to a direct sum decomposition of $\mathbb{R}^{d}$. We omit the proof and note that this is a minor adaptation of that for \cite[Theorem~3.3.5~(ii)]{Call2014}  to the case of weak attractors.
\begin{proposition}
\label{direct sum decomp of attractor/repeller}
    Let $(A,R)$ be a weak attractor-repeller pair for $\mathbb{P}\Phi$. Then $\mathbb{P}^{-1}A(\omega)$ and $\mathbb{P}^{-1}R(\omega)$ are almost surely linear subspaces of $\mathbb{R}^{d}$. Moreover we have that $\mathrm{dim}(\mathbb{P}^{-1}A(\omega))$ and $\mathrm{dim}(\mathbb{P}^{-1}R(\omega))$ are almost surely constant and that $$\mathbb{P}^{-1}A(\omega)\oplus \mathbb{P}^{-1}R(\omega)=\mathbb{R}^{d}  \text{ almost surely. }$$
\end{proposition}
Using very similar arguments as in \cite[Theorem~3.3.6]{Call2014}, one finally obtains the following theorem.
\begin{theorem}
\label{Theorem unique finest Morse decomposition}
    $\mathbb{P}\Phi$ has a unique finest Morse decomposition $\{M_{1},...,M_{n}\}$. Moreover $n\leq d$ and for almost every $\omega\in \Omega$ we have
    $$\mathbb{P}^{-1}M_{1}(\omega)\oplus...\oplus \mathbb{P}^{-1}M_{n}(\omega)=\mathbb{R}^{d}.$$
\end{theorem}

We often require a bounded growth condition on the norms $\|\Phi(t,\omega)\|$ of the following form: for almost every $\omega\in \Omega$ it holds that $\|\Phi(t,\omega)\|\leq K(\omega)e^{a|t|}$ for some $a>0$, all $t\in\mathbb{T}$ and some tempered random variable $K:\Omega \to [1,\infty)$.

The following lemma tells us that we may assume that the bounded growth condition holds on a $\theta$-invariant full-measure set.
\begin{lemma}
\label{invariant set bg condition}
    Suppose there exists an $a>0$, a tempered (respectively essentially bounded)  random variable $K:\Omega\to [1,\infty)$ and a full-measure set $F\in \mathcal{F}$ such that for any $\omega \in \mathcal{F}$ one has that 
    \begin{align}
    \label{proof technical lemma finitness}
    \|\Phi(t,\omega)\|\leq K(\omega)e^{a|t|} \text{ for all } t\in\mathbb{T}.
    \end{align}
    Then, for any $\varepsilon>0$ there exists a $\theta$-invariant full-measure set $\tilde{F}\in \mathcal{F}$ such that \eqref{proof technical lemma finitness} holds with $a>0$ replaced by $a+\varepsilon>0$ for all $\omega\in \tilde{F}$ and some tempered (respectively essentially bounded) random variable $\tilde{K}:\Omega \to [1,\infty)$. Moreover for any $\omega\in \tilde{F}$, and any bounded sequence $(s_{n})_{n\in\mathbb{N}}$ one has $\limsup_{n\to\infty} \tilde{K}(\theta_{s_{n}}\omega)<\infty$.
\end{lemma}

\begin{proof}
Let $B\in \mathcal{F}$ be the $\theta$-invariant full-measure set where $\lim_{t\to \pm \infty} \frac{1}{|t|}\ln(K(\theta_{t}\omega))=0$ holds and consider the set $$\hat{F}:=\{\omega\in \Omega: \liminf_{t\to \pm \infty} \frac{1}{|t|} \ln\|\Phi(t,\omega)\|\geq -a, \limsup_{t\to \pm \infty} \frac{1}{|t|} \ln\|\Phi(t,\omega)\|\leq a\}.$$
Since one has for any $s,t\in \mathbb{T}$ and $\omega\in \Omega$ that $\Phi(t,\theta_{s}\omega)=\Phi(t+s,\omega)\Phi(s,\omega)^{-1}$ one can check that $\hat{F}$ is $\theta$-invariant. Moreover $\hat{F}$ has full-measure since it contains $\bigcap_{t\in\mathbb{T}\cap\mathbb{Q}}\theta_{t}(F)\cap B $. In order to see this first note that for $\omega\in \bigcap_{t\in\mathbb{T}\cap\mathbb{Q}}\theta_{t}(F)\cap B $ one has that $t\mapsto \|\Phi(-t,\theta_{t}\omega)\|$ is continuous as $\Phi(-t,\theta_{t}\omega)=\Phi(t,\omega)^{-1}$ and matrix inversion is continuous on $\text{GL}(d,\mathbb{R})$. Now take a $\omega\in \bigcap_{t\in\mathbb{T}\cap\mathbb{Q}}\theta_{t}(F)\cap B$, then as $\|\Phi(t,\omega)\|\leq K(\omega)e^{a|t|}$ for any $t\in \mathbb{T}$ it follows that one must have $\limsup_{t\to \pm \infty} \frac{1}{|t|} \ln\|\Phi(t,\omega)\|\leq a $. Let now $(t_{n})_{n\in\mathbb{N}}$ with $\lim_{n\to\infty}t_{n} \in \{-\infty,\infty\}$, then for any $\varepsilon>0$ there exists a sequence $(t_{n}^{\prime})_{n\in\mathbb{N}}\subset \mathbb{Q}\cap \mathbb{T}$ such that $\lim_{n\to\infty}(t_{n}-t_{n}^{\prime})=0$ and $\|\Phi(-t_{n}^{\prime},\theta_{t_{n}^{\prime}}\omega)\|+\varepsilon \geq  \|\Phi(-t_{n},\theta_{t_{n}}\omega)\|$ for any $n\in \mathbb{N}$. It follows  that :
$$\|\Phi(t_{n},\omega)\|\geq \frac{1}{\|\Phi(-t_{n},\theta_{t_{n}}\omega)\|}\geq \frac{1}{\|\Phi(-t_{n}^{\prime},\theta_{t_{n}^{\prime}}\omega)\|+\varepsilon}\geq \frac{1}{K(\theta_{t_{n}^{\prime}}\omega)e^{a|t_{n}^{\prime}|}+\varepsilon}. $$
This implies that $\liminf_{t\to \pm \infty} \frac{1}{|t|} \ln\|\Phi(t,\omega)\|\geq -a $.
\\
Consider the $\theta$-invariant full-measure set $\tilde{F}:=\hat{F}\cap B$
and define for an arbitrary $\varepsilon>0$ the random variable $\tilde{K}:\Omega \to [1,\infty]$ given by $K(\omega)=1$ for $\omega\in \Omega\setminus \tilde{F}$
and for $\omega\in \tilde{F}$:
$$ \tilde{K}(\omega )=  \inf \{K\geq 1: \forall t\in \mathbb{T}: \ \|\Phi(t,\omega)\|\leq Ke^{(a+\varepsilon) |t|}\}. $$
Since $\mathcal{F}$ is complete and $t\mapsto \|\Phi(t,\omega)\|$ is continuous, it is easy to see that $\tilde{K}$ is indeed a random variable. We now demonstrate that for any $\varepsilon>0$, $\tilde{K}$ is real-valued and tempered, this will then conclude the proof. Suppose first that $\tilde{K}$ is not real-valued for some $\omega\in \tilde{F}$, then there exists a sequence of times $(t_{n})_{n\in\mathbb{N}}\subset \mathbb{T}$ such that for any $n\in\mathbb{N}$ one has that $\|\Phi(t_{n},\omega)\|>ne^{(a+\varepsilon)|t_{n}|}$. This immediately yields a contradiction to the definition of $\tilde{F}$. Now suppose that $\tilde{K}$ is not tempered, then there exists a $\omega\in \tilde{F}\cap \bigcap_{t\in\mathbb{Q}\cap \mathbb{T}}\theta_{t}(F$) such that $\limsup_{t\to \pm \infty}\frac{1}{|t|}\ln (\tilde{K}(\theta_{t}\omega))=\infty$. Let $s_{n}\uparrow \infty$ be a sequence where the $\limsup$ is attained. Then by definition of $\tilde{K}$, for any $n\in\mathbb{N}$ there exists a $t_{n}\in \mathbb{T} $ such that $\|\Phi(t_{n},\theta_{s_{n}}\omega)\|>\frac{1}{2}\tilde{K}(\theta_{s_{n}}\omega)e^{(a+\varepsilon)|t_{n}|}$. By continuity we may select a sequence of times $(s_{n}^{\prime})_{n\in\mathbb{N}}\subset \mathbb{Q}\cap \mathbb{T}$ such that $\lim_{n\to\infty} \frac{s_{n}^{\prime}}{s_{n}}=1$ and $\|\Phi(t_{n},\theta_{s_{n}^{\prime}}\omega)\|>\frac{1}{2}\tilde{K}(\theta_{s_{n}}\omega)e^{(a+\varepsilon)|t_{n}|}$. This in turn implies that one has $K(\theta_{s_{n}^{\prime}})e^{a|t_{n}|}>\frac{1}{2}\tilde{K}(\theta_{s_{n}}\omega)e^{(a+\varepsilon)|t_{n}|}$ for every $n\in\mathbb{N}$, but on the one hand one has that 
$\limsup_{n\to\infty} \frac{1}{s_{n}}\ln(K(\theta_{s_{n}^{\prime}}\omega))=\limsup_{n\to\infty} \frac{1}{s_{n}^{\prime}}\ln(K(\theta_{s_{n}^{\prime}}\omega))=0$ and on the other hand $\liminf_{n\to\infty} \frac{1}{s_{n}}\ln(\tilde{K}(\theta_{s_{n}}\omega))=\infty$. This is a contradiction and hence demonstrates temperedness of $\tilde{K}$. It remains to show that for any $\omega\in \tilde{F}$, and any bounded sequence $(s_{n})_{n\in\mathbb{N}}\subset \mathbb{T}$ one has that $\limsup_{n\to\infty} \tilde{K}(\theta_{s_{n}}\omega)<\infty$. To that aim, let $\omega\in \mathcal{F}$, $(s_{n})_{n\in\mathbb{N}}\subset \mathbb{T}$ be bounded, and suppose that $(\tilde{K}(\theta_{s_{n}}))_{n\in \mathbb{N}}$ is unbounded and strictly increasing. Then there exists a sequence $(t_{n})_{n\in \mathbb{N}}$ with $\|\Phi(t_{n},\theta_{s_{n}}\omega)\|>\frac{1}{2}\tilde{K}(\theta_{s_{n}}\omega)e^{(a+\varepsilon)|t_{n}|}$. This implies that $\limsup_{n\to\infty} |t_{n}|=\infty$. We may pass now  (without relabeling) to a subsequence with $\lim_{n\to\infty}t_{n}\in \{-\infty,\infty \}$. Since we in particular have that $\|\Phi(t_{n}+s_{n},\omega)\|\cdot \|\Phi(-s_{n},\theta_{s_{n}}\omega)\|>\frac{1}{2}\tilde{K}(\theta_{s_{n}}\omega)e^{(a+\varepsilon)|t_{n}|}$ and $(s_{n})_{n\in\mathbb{N}}$ converges to some element $s\in \mathbb{T}$ (after passing to an appropriate subsequence), this implies that $\liminf_{n\to\infty} \frac{1}{|s_{n}+t_{n}|}\ln \|\Phi(t_{n}+s_{n},\omega)\|\geq a+\varepsilon$, which is a contradiction to $\omega\in \hat{F}$.

If $K$ is essentially bounded one sees by a similar argumentation that $\tilde{K}$ needs to be essentially bounded as well.
\end{proof}

We now define the Morse spectrum following \cite[Section~3.4]{Call2014}.
\begin{definition}
    Let $M\subset\mathbb{P}^{d-1}$ be an invariant non-trivial compact random set, such that $\mathbb{P}^{-1}M(\omega)$ is a linear subspace of $\mathbb{R}^{d}$ almost surely and define:
    \begin{align*}
\Xi(\mathbb{P}^{-1}M)(\omega) := & \{ \xi \in \overline{\mathbb{R}} :  \text{there exists a sequence } \{(T_k, t_k, x_k)\}_{k \in \mathbb{N}} \text{ with } T_k, t_k \in \mathbb{T}, \\& \ |T_{k}|\geq |t_{k}|,  x_k \in \mathbb{P}^{-1} M(\theta_{t_k}\omega)\setminus\{0\} \nonumber 
 \text{ such that } \lim_{k \to \infty} T_k = \infty \\& \text{ and } \lim_{k \to \infty} \lambda^{T_k}(\theta_{t_k}\omega, x_k) = \xi\},
\end{align*}
where $\lambda^T(\omega,x)$ is the finite-time Lyapunov exponent defined by
\[
\lambda^T(\omega,x) := \frac{1}{T} \ln \frac{\|\Phi(T, \omega)x\|}{\|x\|}.
\]
It is convenient to define 
\[
\tilde{\lambda}(T,t,\omega,x) := \frac{1}{T} \ln \frac{\|\Phi(T+t, \omega)x\|}{\|\Phi(t,\omega)x\|}.
\].

Finally given a unique finest weak Morse decomposition $\{M_{1},\dots,M_{n}\}$, $n\leq d$, we define the Morse spectrum to be
$$\Xi_{w}(\omega)=\bigcup_{i=1}^{n}\Xi(\mathbb{P}^{-1}M_{i})(\omega).$$ 
\end{definition}
\begin{remark}
There are several reasons why we demand that $|T_{k}|\geq |t_{k}|$ for the sequences $(T_{k},t_{k})_{n\in \mathbb{N}}$ in the definition of the Morse spectrum.\begin{itemize}
\item[(i)] If we assume that the bounded growth condition \eqref{proof technical lemma finitness} holds, one obtains that $\frac{1}{T_{k}} \ln \|\Phi(T_{k},\theta_{t_{k}}\omega)\|\leq \frac{1}{T_{k}}\ln K(\theta_{t_{k}}\omega)+a$. In order for this bound to be meaningful we would need that $(\frac{1}{T_{k}} \ln K(\theta_{t_{k}}\omega))_{k\in \mathbb{N}}$ does not grow to infinity; if we allow for arbitrary sequences $(t_{k})_{k\in\mathbb{N}}$ this can not be guaranteed, however, if $|T_{k}|\geq |t_{k}|$ one would have that $(\frac{1}{T_{k}} \ln K(\theta_{t_{k}}\omega))_{k\in \mathbb{N}}$ tends to zero.
\item[(ii)] A property one would like to have is that $\Xi_{w}$ is invariant under tempered coordinate change. This is not the case if we allow for arbitrary sequences $(t_{k})_{k\in \mathbb{N}}$. An example will be given at the end of this subchapter (cf. Example \ref{example motivation for alternativ def morse}).
\item[(iii)] It will turn out that this set-up offers a constructive alternative to the non-uniform dichotomy spectrum (c.f. Theorem \ref{morse spectrum equals dichotomy spectrum}). 
\end{itemize}
\end{remark}
\begin{remark}
    If we assume that $\Phi$ fulfills the Oseledet integrability conditions, one has that the Lyapunov spectrum  $\Lambda$ is contained in $\Xi_{w}$. This follows immediately from the fact that the projected Oseledet subspaces define a weak Morse decomposition (cf.~\cite[Theorem 6.1]{crauel2004towards}).
\end{remark}

The following two technical lemmas will be useful. The first one is taken from \cite[Corollary 4.7]{viana2014lectures}, and the second one comes from \cite[Corollary 2.13]{crauel2002random}.
\begin{lemma}
\label{random sets under preimage}
    Let $M(\omega)\subset \mathbb{P}^{d-1}$ be a random set. Then $\mathbb{P}^{-1}M(\omega)$ is a random set as well.
\end{lemma}
\begin{lemma}
    \label{measurability lemma randomprobmas}
    Let $(\Omega,\mathcal{F},\mu)$ be a complete probability space, $X$ be a Polish space (equipped with its Borel $\sigma$-Algebra) and $f:\Omega\times X\to \mathbb{R}$ be measurable. Then if $C(\omega)\subset X$ is a random set one has that $$\Omega \to \mathbb{R}, \omega\mapsto \sup_{x\in C(\omega)}f(x,\omega)$$ is measurable.
\end{lemma}
The following theorem describes the basic properties of $\Xi$; the proof is an adaptation from the non-autonomous case in \cite[Theorem 5.3]{rasmussen2007alternative}.
\begin{theorem}
\label{basic props of morse}
    Suppose that $\mathbb{T}=\mathbb{R}$, then given an invariant random non-trivial compact set $M\subset \mathbb{P}^{d-1} $ such that $\mathbb{P}^{-1}M(\omega)$ is almost surely a linear subspace of $\mathbb{R}^{d}$ it holds that $\Xi(\mathbb{P}^{-1}M)(\omega)=[a,b]$ almost surely with $-\infty\leq a\leq b\leq \infty$. Furthermore, these statements remain valid for $\mathbb{T}=\mathbb{Z}$ under the assumption that \eqref{proof technical lemma finitness} holds.
\end{theorem}
\begin{proof}
    The proof is divided into two steps: we first show that $\Xi(\mathbb{P}^{-1}M)(\omega)$ is a non-empty closed interval in $\overline{\mathbb{R}}$, and second, we use ergodicity to conclude that $\Xi(\mathbb{P}^{-1}M)(\omega)$ is almost surely constant. In the first step we treat the cases $\mathbb{T}=\mathbb{Z}$ and $\mathbb{T}=\mathbb{R}$ separately.
\\
\emph{Step 1:} It follows immediately from the definition via the limits that $\mathbb{P}^{-1}M(\omega)$ is non-empty and closed for any $\omega\in \Omega$ such that $\mathbb{P}^{-1}M(\omega)\neq \{0\}$. We show now that  $\Xi(\mathbb{P}^{-1}M)(\omega)$ is an interval almost surely. First note that since $\Phi(T,\theta_{t}\omega)=\Phi(T+t,\omega)\Phi(-t,\theta_{t}\omega)$ for any $T,t\in\mathbb{T}$ and $\omega\in \Omega$ we may write 
 \begin{align*}
\Xi(\mathbb{P}^{-1}M)(\omega)  :=&  \{\xi \in \overline{\mathbb{R}} :  \text{there exists a sequence } \{(T_k, t_k, x_k)\}_{k \in \mathbb{N}} \text{ with } T_k, t_k \in \mathbb{T}, \\& \, x_k \in \mathbb{P}^{-1} M(\omega)\setminus\{0\} \nonumber 
 \text{ such that } \lim_{k \to \infty} T_k = \infty, |T_{k}|\geq |t_{k}| \\& \text{ and } \lim_{k \to \infty} \tilde{\lambda}(T_{k},t_{k},\omega),x_{k}) = \xi\}.
\end{align*}

 let $F\in \mathcal{F}$ be the full-measure set where $M$ is invariant and let $\omega\in F$ and $a,b\in \Xi(\mathbb{P}^{-1}M)(\omega)\cap \mathbb{R}$ with $a<b$. Then there exist sequences $((T_{k}^{j},t_{k}^{j},x_{k}^{j})\subset (\mathbb{T},\mathbb{T},\mathbb{P}^{-1}M(\omega)\setminus \{0\})$ with 
  $|T_{k}^{j}|\geq|t_{k}^{j}|$ for $j\in \{a,b\}$ such that $$j=\lim_{k\to\infty}\tilde{\lambda}(T_{k}^{j},t_{k}^{j},\omega,x_{k}^{j}).$$
  We write $y_{k}^{j}=\Phi(t_{k}^{j},\omega)x_{k}^{j}$ for $j\in \{a,b\}$.
 Now let $s\in (a,b)$ and assume w.l.o.g  that for any $k\in\mathbb{N}$ it holds that $\tilde{\lambda}(T_{k}^{a},t_{k}^{a},\omega,x_{k}^{a})<s<\tilde{\lambda}(T_{k}^{b},t_{k}^{b},\omega,x_{k}^{b})$.
\\
\emph{Case 1: $\mathbb{T}=\mathbb{R}$}
\\
For $c\in [0,1]$ we set $$t_{k}^{c}=ct_{k}^{a}+(1-c)t_{k}^{b}, T_{k}^{c}=cT_{k}^{a}+(1-c)T_{k}^{b} \text{ and }x_{k}^{c}=cx_{k}^{a}+(1-c)x_{k}^{b}.$$ (Note that $|T_{k}^{c}|\geq |t_{k}^{c}|)$. Finally define 
$$H_{k}:[0,1]\to \mathbb{R}, c\mapsto \frac{1}{T_{k}^{c}} \ln \frac{\|\Phi(T_{k}^{c}+t_{k}^{c},\omega)x_{k}^{c}\|}{\|\Phi(t_{k}^{c},\omega)x_{k}^{c}\|}=\tilde{\lambda}(T_{k}^{c},t_{k}^{c},\omega,x_{k}^{c}).$$
It is easily checked that $H_{k}$ is continuous. Now since $H_{k}(0)<s$ and $H_{k}(1)>s$ there exists $c_{k}\in (0,1)$ with $H_{k}(c_{k})=s$ and so by setting $\tilde{t}_{k}=t_{k}^{c_{k}}$, $\tilde{T}_{k}=T_{k}^{c_{k}}$ and $\tilde{x}_{k}=x_{k}^{c_{k}}$ we obtain a sequence with the desired property.
\\
\emph{Case 2: $\mathbb{T}=\mathbb{Z}$}:
\\
For a fixed $\omega\in \Omega$ we consider the map 
$$\Phi_{\omega}:\mathbb{T}\times \mathbb{R}^{d} \to \mathbb{R}, \ (T,x)\mapsto \|\Phi(T,\omega)x\|.$$
We want to first extend this map continuously to $\mathbb{R}\times \mathbb{R}^{d}$. To that aim we set for any $T\in\mathbb{R}$ and $x\in\mathbb{R}^{d}$:
$$\Phi_{\omega}(T,x):=(1+\floor*{T}-T)\Phi_{\omega}(\floor*{T},x)+(T-\floor*{T})\Phi_{\omega}(\ceil*{T},x).$$
It is easily checked that this defines a continuous extension. Now, 
arguing similarly as before one obtains sequences $(\tilde{t}_{k})_{k\in\mathbb{N}}\subset \mathbb{R}$, $(\tilde{T}_{k})_{k\in\mathbb{N}}\subset \mathbb{R}$ and $(\tilde{x}_{k})_{k\in\mathbb{N}}\subset \mathbb{P}^{-1}M(\omega)\setminus \{0\}$ such that 
$$|\tilde{T}_{k}|\geq |\tilde{t}_{k}| \text{ and }\frac{1}{\tilde{T}_{k}} \ln \frac{\Phi_{\omega}(\tilde{T}_{k}+\tilde{t}_{k},\tilde{x}_{k})}{\Phi_{\omega}(\tilde{t}_{k},\tilde{x}_{k})}=s \text{ for any } k\in\mathbb{N}.$$ Now for any $(T,x)\in \mathbb{R}\times \mathbb{R}^{d}$, condition \eqref{bg in Prop finitness} yields the estimations 
$$ \frac{\Phi_{\omega}(T,x)}{e^{a}K(\theta_{\floor*{T}}\omega)}\leq \Phi_{\omega}(\floor*{T},x)\leq e^{a}K(\theta_{\ceil*{T}}\omega)\Phi_{\omega}(T,x).$$
This, and temperedness of $K$ imply that 
$$\lim_{k\to\infty} \frac{1}{\floor*{\tilde{T}_{k}}} \ln \frac{\Phi_{\omega}(\floor*{\tilde{T}_{k}}+\floor*{\tilde{t}_{k}},\tilde{x}_{k})}{\Phi_{\omega}(\floor*{\tilde{t}_{k}},\tilde{x}_{k})}=s.$$
\\
\emph{Step 2:}
It is easy to see that for any $\omega\in \Omega$ and $s\in\mathbb{T}$ one has that $\Xi(\mathbb{P}^{-1}M)(\omega)=\Xi(\mathbb{P}^{-1}M)(\theta_{s}\omega)$.
Now, using the fact that $\Xi(\mathbb{P}^{-1}M)(\omega)$ is almost surely a closed interval, it suffices by ergodicity to show that the maps $$\eta(\omega):=\sup \Xi(\mathbb{P}^{-1}M)(\omega)  \text{ and }\alpha(\omega):=\inf \Xi(\mathbb{P}^{-1}M)(\omega) $$ are measurable. One sees by using the alternative definition of the Morse spectrum that 
$$\eta(\omega)=\max_{\tilde{t}\in \{-\infty,0,\infty\}} \ \limsup_{T\to\infty, \ t\to \tilde{t}, \ |T|\geq|t|} \ \ \sup_{x\in \mathbb{P}^{-1}M(\omega)\cap \mathbb{S}^{d-1}} \frac{1}{T}\ln\frac{\|\Phi(T+t,\omega)x\|}{\|\Phi(t,\omega)x\|}.$$
Now since the map $$\mathbb{T}\times \mathbb{T}\times (\mathbb{P}^{-1}M(\omega)\cap \mathbb{S}^{d-1})\to \mathbb{R}, \ (T,t,x)\mapsto \frac{1}{T}\ln\frac{\|\Phi(T+t,\omega)x\|}{\|\Phi(t,\omega)x\|} $$
is continuous for any $\omega\in \Omega$, and $\mathbb{P}^{-1}M(\omega)\cap \mathbb{S}^{d-1}$ is compact almost surely, it follows that the map
$$\mathbb{T}\times \mathbb{T}\to \mathbb{R}, \ (T,t)\mapsto \sup_{x\in \mathbb{P}^{-1}M(\omega)\cap \mathbb{S}^{d-1}}\frac{1}{T}\ln\frac{\|\Phi(T+t,\omega)x\|}{\|\Phi(t,\omega)x\|}. $$
is continuous almost surely. Moreover by Lemma \ref{measurability lemma randomprobmas} and Lemma \ref{random sets under preimage} we have that
for any fixed $T,t\in \mathbb{T}$ that the map
$$\Omega \to \mathbb{R}, \ \omega\mapsto \sup_{x\in \mathbb{P}^{-1}M(\omega)\cap \mathbb{S}^{d-1}}\frac{1}{T}\ln\frac{\|\Phi(T+t,\omega)x\|}{\|\Phi(t,\omega)x\|} $$
is measurable. The measurability of $\eta$ follows now from standard measurability arguments. We can apply analogous arguments to show that $\alpha$ is measurable.
\end{proof}

Now we state and prove a necessary and sufficient condition for $\Xi$ to be finite.

\begin{proposition}
    \label{Proposition finitness}
     $\Xi(\mathbb{R}^{d})(\omega)$ is finite almost surely  if and only if 
    there exists a $\theta$-invariant full-measure set $F\in\mathcal{F}$, a tempered random variable $K:\Omega \to [1,\infty)$ and an $\alpha>0$ such that for any $\omega\in F$
    \begin{align}
    \label{bg in Prop finitness}
    \|\Phi(t,\omega)\|\leq K(\omega)e^{\alpha|t|} \text{ for all }t\in\mathbb{T}.
    \end{align}
\end{proposition}
\begin{proof}
Assume that \eqref{bg in Prop finitness} holds then we may assume that \eqref{bg in Prop finitness} holds on a $\theta$-invariant full-measure set $F\in \mathcal{F}$ and we may assume that for any $\omega \in F$, $(K(\theta_{t_{n}}\omega))_{n\in\mathbb{N}}$ is bounded whenever $(t_{n})_{n\in\mathbb{N}}\subset \mathbb{T}$ is bounded (cf. Lemma \ref{invariant set bg condition}). Let $\omega\in F$ and take sequences $(T_{k},t_{k})_{k\in\mathbb{N}}\subset \mathbb{T}\times \mathbb{T}$ with $T_{k}\uparrow \infty$ and $|T_{k}|\geq |t_{k}|$ for any $k\in \mathbb{N}$
then $$\limsup_{k\to \infty} \frac{1}{T_{k}} \ln \|\Phi(T_{k},\theta_{t_{k}}\omega)\|\leq \alpha+ \limsup_{k\to \infty} \frac{1}{T_{k}} \ln K(\theta_{t_{k}}\omega)= \alpha.$$
Where the last equality follows from $|T_{k}|\geq |t_{k}|$ and temperedness of $K$.
\\
Moreover, for an arbitrary sequence $(x_{k})_{k\in\mathbb{N}}\subset \mathbb{S}^{d-1}$  we have:
$$\liminf_{k\to \infty} \frac{1}{T_{k}} \ln \|\Phi(T_{k},\theta_{t_{k}}\omega)x_{k}\|\geq -\alpha- \limsup_{k\to \infty} \frac{1}{T_{k}} \ln K(\theta_{t_{k}+T_{k}}\omega)= - \alpha.$$
If $\limsup_{k\to\infty} |T_{k}+t_{k}|=\infty$ the last equality follows from $|T_{k}|\geq |t_{k}|$ and temperedness of $K$, if $\limsup_{k\to\infty} |T_{k}+t_{k}|<\infty$, the equality follows from the fact that $(K(\theta_{t_{n}+T_{n}}\omega))_{n\in\mathbb{N}}$ is bounded.
\\
 Now assume conversely that $\Xi(\mathbb{R}^{d})(\omega)$ is finite almost surely, then by Theorem \ref{basic props of morse} there exists a $\theta$-invariant full-measure set $F\in\mathcal{F}$ and a number $a>0$ such that $\Xi(\mathbb{R}^{d})(\omega)\subset [-a,a]$ for all $\omega\in F$. For an arbitrary $\varepsilon>0$ consider the random variable $K:\Omega \to [1,\infty]$ defined by $K(\omega)=1$ for $\omega\in \Omega\setminus F$
and for $\omega\in F$:
$$ K(\omega )=  \inf \{K\geq 1: \forall t\in\mathbb{T}, \ \|\Phi(t,\omega)\|\leq Ke^{(a+\varepsilon) |t|}\}. $$
Since $\mathcal{F}$ is complete and $t\mapsto \|\Phi(t,\omega)\|$ is continuous, it is easy to see that $K$ is indeed a random variable.
\\
Now we show that $K$ is real-valued and tempered which then finishes the proof. Suppose first that there exists an $\omega\in F$ such that $K$ is not real valued. Then there exists a sequence $(t_{n})_{n\in\mathbb{N}}\subset \mathbb{T}$ and a sequence $(x_{n})_{n\in\mathbb{N}}\subset\mathbb{S}^{d-1} $ such that for any $n\in\mathbb{N}$ we have $\|\Phi(t_{n},\omega)x_{n}\|>ne^{(a+\varepsilon)|t_{n}|}$. This implies that $\limsup_{n\to\infty}|t_{n}|=\infty$, hence we may extract a subsequence (which we label w.l.o.g again by $(t_{n})_{n\in\mathbb{N}}$) such that $\lim_{n\to\infty}t_{n}\in \{-\infty,\infty\}$. If $\lim_{n\to\infty}t_{n}=\infty$ we immediately see that $a+\varepsilon\leq \sup\Xi(\mathbb{R}^{d})(\omega)$. This is a contradiction. If $\lim_{n\to\infty}t_{n}=-\infty$ we set $y_{n}=\Phi(t_{n},\omega)x_{n}$ then 
$$\lambda^{-t_{n}}(\theta_{t_{n}}\omega,y_{n})=\frac{1}{|t_{n}|}\ln \frac{1}{\|y_{n}\|}<-\frac{1}{|t_{n}|}(\ln(n)+(a+\varepsilon)|t_{n}|)\leq -\frac{1}{|t_{n}|}((a+\varepsilon)|t_{n}|) $$
which implies that $-(a+\varepsilon)\geq \inf\Xi(\mathbb{R}^{d})(\omega)$ and again gives us a contradiction.
\\
Next we demonstrate temperedness of $K$. Assume to the aim of a contradiction that $K$ is not tempered
then there exists a $\omega\in F$ and  a sequence $t_{n}\uparrow \infty$ such that \\ $\limsup_{n\to \infty}\frac{1}{t_{n}}\ln(K(\theta_{t_{n}}\omega))=\infty$. Now for any $n\in\mathbb{N}$ there exists by definition of $K$ a $s_{n}\in\mathbb{T}$  such that $$\|\Phi(s_{n},\theta_{t_{n}}\omega)\|>\frac{1}{2}K(\theta_{t_{n}}\omega)e^{(a+\varepsilon) |s_{n}|}.$$ 
If $|s_{n}|\geq |t_{n}|$ for infinitely many $n\in\mathbb{N}$, we obtain by arguing similarly to before that $\{-(a+\varepsilon),a+\varepsilon\}\cap \Xi(\mathbb{R}^{d})(\omega)\neq \emptyset$ which gives a contradiction. Now assume w.l.o.g that $|t_{n}|\geq |s_{n}|$ for any $n\in\mathbb{N}$. Then we obtain that for any $n\in \mathbb{N}$ that $$\|\Phi(s_{n}+t_{n},\omega)\|\cdot\|\Phi(t_{n},\omega)^{-1}\|>\frac{1}{2}K(\theta_{t_{n}}\omega)e^{(a+\varepsilon) |s_{n}|}$$
which implies that $\limsup_{n\to\infty} \frac{1}{|t_{n}|} \ln \|\Phi(s_{n}+t_{n},\omega)\|=\infty  $ or \\ $\limsup_{n\to\infty} \frac{1}{|t_{n}|} \ln \|\Phi(t_{n},\omega)^{-1}\|=\infty $. By arguing similarly to before one obtains  that $\{-\infty,\infty\}\cap \Xi(\mathbb{R}^{d})(\omega)\neq \emptyset$.
 This is a contradiction and finishes the proof. 
\end{proof}

 The following proposition states that the weak Morse spectrum is invariant under tempered random linear coordinate changes.
\begin{proposition}
\label{Prop morse spectrum coordinate change}
Let $H:\Omega\to \text{GL}(d,\mathbb{R})$ be measurable and such that $\mathbb{T}\to  \text{GL}(d,\mathbb{R}) $, $t\mapsto H(\theta_{t}\omega)$ is continuous for any $\omega\in \Omega$. Further suppose that the random variables $\omega \mapsto \|H(\omega)^{\pm 1}\|$ are tempered i.e 
$$\lim_{t\to\pm\infty}\frac{1}{|t|}\ln\|H(\theta_{ t}\omega)^{\pm 1}\|=0 \text{ almost surely}.$$
Then the weak Morse spectrum $\Xi_{\Psi}$ obtained from the linear random dynamical system $$\Psi(t,\omega):=H(\theta_{t}\omega)\Phi(t,\omega)H^{-1}(\omega)$$ agrees with the weak Morse spectrum $\Xi_{\Phi}$ obtained from $\Phi$ almost surely.
\end{proposition}

\begin{proof}
 Let $\{M_{1},...,M_{n}\}$ be the unique finest weak Morse decomposition of $\mathbb{P}\Phi$. Then, the unique finest weak Morse decomposition of $\mathbb{P}\Psi$ is given by $$\{\tilde{M}_{1},...,\tilde{M}_{n}\}=\{\mathbb{P}H(\mathbb{P}^{-1}M_{1}),...,\mathbb{P}H(\mathbb{P}^{-1}M_{n})\}.$$
 We use the equivalent definition for the Morse spectrum from Theorem \ref{basic props of morse}, namely that \begin{align*}
\Xi(\mathbb{P}^{-1}M_{i})(\omega) := & \{\xi \in \overline{\mathbb{R}} :  \text{there exists a sequence } \{(T_k, t_k, x_k)\}_{k \in \mathbb{N}} \text{ with } T_k, t_k \in \mathbb{T}, \\& \, x_k \in \mathbb{P}^{-1} M_{i}(\omega)\setminus \{0\} \nonumber 
 \text{ such that } \lim_{k \to \infty} T_k = \infty, \ |T_{k}|\geq |t_{k}| \\& \text{ and } \lim_{k \to \infty} \frac{1}{T_{k}}\ln \frac{\|\Phi(T_{k}+t_{k},\omega)x_{k}\|}{\|\Phi(t_{k},\omega)x_{k}\|}=\tilde{\lambda}(T_{k},t_{k},\omega,x_{k})=\xi\}.
\end{align*} 
 Let $F\in \mathcal{F}$ be the full-measure set where both $\Xi_{\Phi}$ and $\Xi_{\Psi}$ are constant, and $t\mapsto H(\theta_{t}\omega)$ is continuous. By symmetry it suffices to show that $\Xi_{\Psi}(\omega)\subset \Xi_{\Phi}(\omega)$ for any $\omega\in F$. 
 \\
 Let $\omega\in F$, $i\in\{1,...,n\}$ and $\alpha\in \Xi_{\Phi,i}(\omega):=\Xi_{\Phi}(\mathbb{P}^{-1}M_{i})(\omega)$ then there exist sequences $\{(T_k, t_k, x_k)\}_{k \in \mathbb{N}}$  with  $T_k, t_k \in \mathbb{T}, \, x_k \in \mathbb{P}^{-1} M_{i}(\omega)\setminus \{0\}$
 such that $\lim_{k \to \infty} T_k = \infty$ and $\lim_{k \to \infty} \frac{1}{T_{k}}\ln\frac{\|\Phi(T_{k}+t_{k},\omega)x_{k}\|}{\|\Phi(t_{k},\omega)x_{k}\|}= \alpha$. Now we have that for $k\in\mathbb{N}$, $y_{k}:=H(\omega)x_{k}\in \mathbb{P}^{-1}\tilde{M_{i}}(\omega)\setminus \{0\}$. We obtain
 \begin{align*}
 \frac{\|\Psi(T_{k}+t_{k},\omega)y_{k}\|}{\|\Psi(t_{k},\omega)y_{k}\|} & =\frac{\|H(\theta_{T_{k}+t_{k}}\omega)\Phi(T_{k}+t_{k},\omega)x_{k}\|}{\|H(\theta_{t_{k}}\omega)\Phi(t_{k},\omega)x_{k}\|} \\& \leq \frac{\|\Phi(T_{k}+t_{k},\omega)x_{k}\|}{\|\Phi(t_{k},\omega)x_{k}\|}\cdot\|H(\theta_{T_{k}+t_{k}}\omega)\|\cdot \|H(\theta_{t_{k}}\omega)^{-1}\|.
 \end{align*}
 Similarly:
 \begin{align*}
 \frac{\|\Psi(T_{k}+t_{k},\omega)y_{k}\|}{\|\Psi(t_{k},\omega)y_{k}\|} & =\frac{\|H(\theta_{T_{k}+t_{k}}\omega)\Phi(T_{k}+t_{k},\omega)x_{k}\|}{\|H(\theta_{t_{k}}\omega)\Phi(t_{k},\omega)x_{k}\|} \\& \geq \frac{\|\Phi(T_{k}+t_{k},\omega)x_{k}\|}{\|\Phi(t_{k},\omega)x_{k}\|}\cdot\|H(\theta_{T_{k}+t_{k}}\omega)^{-1}\|^{-1}\cdot \|H(\theta_{t_{k}}\omega)\|^{-1}.
\end{align*}
 We obtain that
 $$\tilde{\lambda}_{\Psi}(T_{k},t_{k},\omega,y_{k})\leq \tilde{\lambda}_{\Phi}(T_{k},t_{k},\omega,x_{k})+\frac{1}{T_{k}}(\ln\|H(\theta_{T_{k}+t_{k}}\omega)\|\cdot\|H(\theta_{t_{k}}\omega)^{-1}\|)$$
 Using $|T_{k}|\geq |t_{k}|$ and temperedness of $\|H(\cdot)\|$ we see that  
 $\limsup_{k\to\infty}\tilde{\lambda}_{\Psi}(T_{k},t_{k},\omega,y_{k})\leq \alpha$.
 Similarly 
 $$\tilde{\lambda}_{\Psi}(T_{k},t_{k},\omega,y_{k})\geq \tilde{\lambda}_{\Phi}(T_{k},t_{k},\omega,x_{k})-\frac{1}{T_{k}}(\ln\|H(\theta_{T_{k}+t_{k}}\omega)^{-1}\|^{-1}\cdot\|H(\theta_{t_{k}}\omega)\|^{-1})$$
 implies that
  $\liminf_{k\to\infty}\tilde{\lambda}_{\Psi}(T_{k},t_{k},\omega,y_{k}) \geq \alpha$. 
  \end{proof}
 
  The following example demonstrates that if one allows for arbitrary sequences $(t_{k})_{k\in\mathbb{N}}$ in the definition of the Morse spectrum, then it is not invariant under tempered coordinate change.
 \begin{example}
 \label{example motivation for alternativ def morse}
    Consider $(\Omega, \mathcal{F}, \mu) = ([0,1), \mathcal{B}([0,1)), \lambda|_{[0,1)})$ and 
$\theta: \Omega \to \Omega$,  $\theta(\omega) = \omega + \alpha \mod 1$ with $\alpha \in \mathbb{R}\setminus\mathbb{Q}$. Furthermore we equip $[0,1)$ with the metric which makes it into a circle, i.e. $d(x,y)=\min\{|x-y|,1-|x-y|\}$. 
 Let now $U(k)=[1-\frac{1}{2^{k-1}},1-\frac{1}{2^{k}})$, $k\in\mathbb{N}$ and define the random variable $\beta(\omega)=\sum_{k\in\mathbb{N}}k\mathbbm{1}_{U_{k}}(\omega)$.
Define the random dynamical system 
$$\Phi(n,\omega)=1 \text{ and } \psi(n,\omega)=\frac{\beta(\theta_{n}\omega)}{\beta(\omega)},\, n\in \mathbb{Z},\, \omega\in \Omega.$$
Then $\psi$ is obtained from $\Phi$ via a tempered random coordinate change. The temperedness follows since $\ln(\beta)\in L^{1}(\Omega,\mathcal{F},\mu)$.
Let $\omega\in \Omega$ be arbitrary and 
consider a sequence $N_{k}\uparrow \infty$ such that $(\theta_{N_{k}}\omega)_{k\in \mathbb{N}}$ is strictly increasing and such that $\lim_{k\to \infty} \theta_{N_{k}}\omega=0\sim 1$. Now we may find a sequence $(n_{k})_{k\in \mathbb{N}}$ such that $(\theta_{n_{k}}\omega)_{k\in \mathbb{N}}$ is bounded away from one (in $[0,1),|\cdot|$) and $\theta_{N_{k}+n_{k}}\omega \in U(2^{N_{k}^{2}})$  for any $k\in \mathbb{N}$. This can be seen by setting $n_{0}=1$ and for $k\in\mathbb{N}$
$$n_{k}=\inf\{n\geq n_{k-1}: \theta_{n}\omega<1-\theta_{N_{k}}\omega \text{ and }\theta_{n}\omega+\theta_{N_{k}}\omega\in U(2^{N_{k}^{2}})\}.$$
Then the density of $(\theta_{n}\omega)_{n\in\mathbb{N}}$ in $[0,1]$ implies that $n_{k}<\infty$ for any $k\in \mathbb{N}$. Moreover, note that for $n,m\in \mathbb{N}$ with $\theta_{n}\omega+\theta_{m}\omega<1$ one has that $\theta_{n+m}\omega=\theta_{n}\omega+\theta_{m}\omega$ (addition in $\mathbb{R}$).
This demonstrates now that $$\lim_{k\to\infty} \frac{1}{N_{k}} \ln \frac{\beta(\theta_{N_{k}+n_{k}}\omega)}{\beta(\theta_{n_{k}}\omega)}=\infty.$$
By temperedness of $\beta$ it also follows that one must have $\limsup_{k\to\infty} \frac{n_{k}}{N_{k}}=\infty$.
\end{example}
We now provide an example which shows that the temperedness assumption in the above proposition  is necessary for the weak Morse spectrum to be invariant under random coordinate change.
 \begin{example}
 \label{example tempered necessary for invariant morse}
    Let $(\Omega,\mathcal{F},\mu)=(\{0,1\}^{\mathbb{Z}},\mathcal{P}(\{0,1\})^{\otimes \mathbb{Z}},\nu^{\otimes \mathbb{Z}})$ where $\mathcal{P}(\{0,1\})$ is the power set of $\{0,1\}$ and $\nu$ is the measure on $\mathcal{P}(\{0,1\})$ defined by $\nu(\{0\})=\nu (\{1\})=\frac{1}{2}$. Now $\theta:\Omega \to \Omega$, $(\omega_{i})_{i\in \mathbb{Z}}\mapsto (\omega_{i+1})_{i\in \mathbb{Z}}$ defines an ergodic metric dynamical system on $(\Omega,\mathcal{F},\mu)$. Let $(U_{k})_{k\in\mathbb{N}}$ be a partition of $\Omega$ such that for any $k\in\mathbb{N}$ we have $\mu(U_{k})=\frac{6}{\pi^{2}k^{2}}$ and define the random variable 
    $$\beta:\Omega \to [1,\infty), \ \omega\mapsto \sum_{k\in\mathbb{N}}e^{k}\mathbbm{1}_{U_{k}}(\omega).$$
    Then $\ln(\beta)\not\in L^{1}(\Omega,\mathcal{F},\mu)$, which already implies that $\beta$ is non-tempered.
    Now define the linear random dynamical systems on $\mathbb{R}$ via 
    \begin{align*}
\Phi (n,\omega)x = x \text{ and } \psi(n,\omega)x=\frac{\beta(\theta_{n}\omega)}{\beta(\omega)}x \text{ for }x\in\mathbb{R},\, \omega\in \Omega.
\end{align*}
Then $\psi$ is obtained from $\Phi$ by a non-tempered random coordinate change. It is clear that $\Xi_{\Phi,w}=\{0\}$ whereas $\infty \in \Xi_{\psi,w}$.
\end{example}
  We now show that the (non-uniform) dichotomy spectrum $\Sigma^{\prime}$ coincides with the weak  Morse spectrum $\Xi_{w}$ under the assumption that there exists an $a>0$, a tempered random variable $K:\Omega \to [1,\infty)$ and a $\theta$-invariant full-measure set $F\in \mathcal{F}$ such that for any $\omega\in F$, $(K(\theta_{t_{n}}\omega))_{n\in\mathbb{N}}$ is bounded, whenever $(t_{n})_{n\in\mathbb{N}}$ is bounded and 
\begin{align}
\label{bounded growth condition}
\|\Phi(t,\omega)\|\leq K(\omega)e^{|a|t} \text{ for any } t\in\mathbb{T}.
\end{align}
(compare Lemma \ref{invariant set bg condition}). In the following section, $F\in \mathcal{F}$ always denotes the $\theta$-invariant full-measure set where 
\begin{itemize}
    \item[(i)] The conditions above hold.
    \item[(ii)] $\mathbb{P}^{-1}M_{i}(\omega)$ is a linear subspace of $\mathbb{R}^{d}$ for any $i\in \{1,\dots,n\}$ and any $\omega\in F$, where $\{M_{1},\dots,M_{n}\}$, $n\leq d$ is the unique finest weak Morse decomposition of $\mathbb{P}\Phi.$
\end{itemize}

\begin{theorem}
\label{morse spectrum equals dichotomy spectrum}
    If $\Phi$ fulfills \eqref{bounded growth condition}, then $\Xi_{w}=\Sigma^{\prime}$.
\end{theorem}
For two random, complementary subspaces $U(\omega)$ and $V(\omega)$ of $\mathbb{R}^{d}$ we often want the projections with range $U(\omega)$ and null-space $V(\omega)$ to be measurable.
\begin{lemma}
\label{measurability projections}
    Let $\mathbb{P}U(\omega)$ and  $\mathbb{P}V(\omega)$ be random, compact sets, such that for all $\omega\in F$ for some full-measure set $F\in \mathcal{F}$ one has $U(\omega)\oplus V(\omega)=\mathbb{R}^{d}$. Then the map $P:\Omega\to \mathbb{R}^{d\times d}$ defined by $P(\omega)=\mathbbm{1}$ for $\omega\in \Omega\setminus F$ and $P(\omega)=P_{U(\omega),V(\omega)}$ is measurable, where $P_{U(\omega),V(\omega)}$ is the unique projection with range $U(\omega)$ and null-space $V(\omega).$
\end{lemma}

\begin{proof}
Lemma \ref{random sets under preimage} guarantees that $U$ and $V$ are  random sets. By completeness of $\mathcal{F}$ we may assume now w.l.o.g that one has $U(\omega)\oplus V(\omega)=\mathbb{R}^{d}$ for all $\omega\in \Omega$.
\\
Let $A^{\dagger}$ denote the Moore-Penrose inverse of a matrix $A\in \mathbb{R}^{d\times d}$. From \cite{golub1996matrix} p.257 one obtains that 
$$A^{\dagger}=\lim_{n\to\infty}(AA^{T}+\frac{1}{n}\mathbbm{1}_{d\times d})^{-1}A^{T}.$$
Now standard measurability arguments imply that the map 
$\dagger:\mathbb{R}^{d\times d}\to \mathbb{R}^{d\times d}$, $A\mapsto A^{\dagger}$ is measurable.  Corollary 5.6 of \cite{ipsen1995angle} gives us the formula 
$$P(\omega)=(P_{V(\omega)^{\perp}}P_{U(\omega)})^{\dagger}$$
where $P_{V(\omega)^{\perp}}$ and $P_{U(\omega)}$ are the orthogonal projections onto $V(\omega)^{\perp}$ and $U(\omega)$. Thus it suffices to prove that $\omega \mapsto P_{U(\omega)}$ is measurable; measurability of $\omega \mapsto P_{V(\omega)^{\perp}} $ follows analogously since one has that $P_{V(\omega)^{\perp}}=\mathbbm{1}-P_{V(\omega)}$. By the projection theorem we have for $y\in \mathbb{R}^{d}$ that $\|(\mathbbm{1}-P_{U(\omega)})y\|=\dist(y,U(\omega))$ which is measurable by Proposition 2.4 \cite{crauel2002random}. It follows that $\|P_{U(\omega)}y\|^{2}$ is measurable since Pythagoras theorem implies $\|P_{U(\omega)}y\|^{2}=\|y\|^{2}-\dist(y,U(\omega))^{2}$. Since this holds for arbitrary $y\in \mathbb{R}^{d}$, the polarisation identity implies that for any $(x,y)\in \mathbb{R}^{d}\times \mathbb{R}^{d}$ we have that $\langle x,P_{U(\omega)}y \rangle$ is measurable. This implies that the entries of $P_{U(\omega)}$ are measurable, finishing the proof.
\end{proof}

A crucial result for  Theorem~\ref{morse spectrum equals dichotomy spectrum} is that invariant compact random subsets of $\mathbb{P}^{d-1}$ which correspond to subspaces may only approach each other with sub-exponential speed.
\begin{lemma}
    \label{Distance between attractor/repeller is tempered}
    Suppose that $\Phi$ fulfills \eqref{bounded growth condition} and 
let $\{M_{1},...,M_{n}\}$, $n\leq d$ be the unique finest weak Morse decomposition for $\mathbb{P}\Phi$ and $ I:= \{1,...,n\}$. Then for any $J\subset I$ , $\|P(\cdot)\|$ is tempered with $P$ being the projection with range $\bigoplus_{j\in J}\mathbb{P}^{-1}M_{j}$ and null-space $\bigoplus_{j\in I\setminus J}\mathbb{P}^{-1}M_{j}$. Moreover $\lim_{t\to\pm \infty} \frac{1}{|t|}\ln \|P(\theta_{t}\omega)\|=0$ for any $\omega\in F$.
\end{lemma}
\begin{proof}
The proof is done in two steps. Note that $(R,N)=(\mathbb{P}\bigoplus_{j\in J}\mathbb{P}^{-1}M_{j},\mathbb{P}\bigoplus_{j\in I\setminus J}\mathbb{P}^{-1}M_{j})$ define invariant, compact disjoint sets.
\\
\emph{Step 1:} It holds that $\tilde{d}_{\mathbb{P}}(R(\cdot),N(\cdot)):=\inf_{x\in R(\cdot)}\inf_{y\in N(\cdot)}d_{\mathbb{P}}(x,y)$ is tempered:
\\
Assume the contrary then there exists an $\omega\in F$
such that  we have $$\liminf_{t\to \infty}\frac{1}{|t|}\ln \tilde{d}_{\mathbb{P}}(R(\theta_{t}\omega),N(\theta_{t}\omega))=-\infty.$$
Now let $t_{n}\to\infty$ be a sequence along which the $\liminf$ is attained. Then since $R$ and $N$ are $\mathbb{P}\Phi$ invariant there exist sequences $(x_{n})_{n\in\mathbb{N}}\subset R(\omega)$ and $(y_{n})_{n\in\mathbb{N}}\subset N(\omega)$ such that 
$$\lim_{n\to\infty} \frac{1}{t_{n}} \ln d_{\mathbb{P}}(\mathbb{P}\Phi(t_{n},\omega)x_{n},\mathbb{P}\Phi(t_{n},\omega)y_{n})=-\infty$$
which in turn implies that there exist sequences $\hat{x}_{n}\in \mathbb{P}^{-1}R(\omega)\cap \mathbb{S}^{d-1}$, $\hat{y}_{n}\in \mathbb{P}^{-1}N(\omega)\cap \mathbb{S}^{d-1}$, $n\in\mathbb{N}$ such that 
$$\lim_{n\to\infty} \frac{1}{t_{n}} \ln \left\|\Phi(t_{n},\omega)\left(\frac{\hat{x}_{n}}{\|\Phi(t_{n},\omega)\hat{x}_{n}\|}-\frac{\hat{y}_{n}}{\|\Phi(t_{n},\omega)\hat{y}_{n}\|}\right)\right\|=-\infty.$$
Letting $a_{n}=\frac{1}{\|\Phi(t_{n},\omega)\hat{x}_{n}\|}$ and $b_{n}=\frac{1}{\|\Phi(t_{n},\omega)\hat{y}_{n}\|}$ we decompose 
$$\|a_{n}\hat{x}_{n}-b_{n}\hat{y}_{n}\|^{2}=(a_{n}-b_{n}\langle \hat{x}_{n},\hat{y}_{n}\rangle)^{2}+b_{n}^{2}\|\hat{y}_{n}-\langle \hat{x}_{n},\hat{y}_{n}\rangle\hat{x}_{n}\|^{2}\geq C^{2}b_{n}^{2}$$
for $C^{2}=\inf_{n\in\mathbb{N}}\|\hat{y}_{n}-\langle \hat{x}_{n},\hat{y}_{n}\rangle\hat{x}_{n}\|^{2}>0$. Please note that we have that $C>0$ since $\mathbb{P}^{-1}(R(\omega))\cap \mathbb{S}^{d-1}$ and $\mathbb{P}^{-1}(N(\omega))\cap \mathbb{S}^{d-1}$ are compact.
\\
We obtain now
\begin{align*}
\|\Phi(t_{n},\omega)(a_{n}\hat{x}_{n}-b_{n}\hat{y}_{n})\| &\geq \frac{\|a_{n}\hat{x}_{n}-b_{n}\hat{y}_{n}\|}{\|\Phi(-t_{n},\theta_{t_{n}}\omega)\|}\geq  \frac{Cb_{n}}{K(\theta_{t_{n}}\omega)e^{at_{n}}}\\&\geq \frac{C}{K(\theta_{t_{n}}\omega)e^{at_{n}}\|\Phi(t_{n},\omega)\|}\geq  \frac{C}{K(\theta_{t_{n}}\omega)K(\omega)e^{2at_{n}}}.
\end{align*}
This now gives us $$-\infty=\liminf_{n\to\infty} \frac{1}{t_{n}} \ln \|\Phi(t_{n},\omega)(a_{n}\hat{x}_{n}-b_{n}\hat{y}_{n})\|\geq -2a$$
which is a contradiction and finishes the first step.
\\
\emph{Step 2:} $\|P(\cdot)\|$ is tempered:
\\
Let $\gamma(\omega):=\tilde{d}_{\mathbb{P}}(R(\omega),N(\omega))$, then by the first part $\gamma$ is tempered. Let $\omega\in F$ and $(t_{n})_{n\in\mathbb{N}}\subset\mathbb{T}$ with $\lim_{n\to\infty}t_{n}=\infty$. Then writing 
$$a_{n}=\max_{r\in \mathbb{P}^{-1}R(\theta_{t_{n}}\omega)\cap \mathbb{S}^{d-1}} \ \max_{n\in \mathbb{P}^{-1}N(\theta_{t_{n}}\omega)\cap \mathbb{S}^{d-1}}\langle r,n\rangle\in (0,1)$$
we have that (see \cite{ipsen1995angle} Theorem 3.1)
$$\|P(\theta_{t_{n}}\omega)\|=\frac{1}{\sqrt{1-a_{n}^{2}}}.$$
Now suppose that $\limsup_{n\to\infty}\frac{1}{t_{n}}\ln \|P(\theta_{t_{n}}\omega)\|=\infty$, then for the subsequence along which the $\limsup$ is attained (which we label w.l.o.g again by $(t_{n})_{n\in\mathbb{N}}$) we have that $\lim_{n\to\infty}\frac{1}{t_{n}}\ln(1-a_{n}^{2})=-\infty$ and hence also $\lim_{n\to\infty}\frac{1}{t_{n}}\ln(1-a_{n})=-\infty$ .
Since $1-a_{n}\geq \frac{\gamma(\theta_{t_{n}}\omega)^{2}}{2}$ this yields a contradiction to the temperedness of $\gamma$.
\end{proof}

The following lemma is an adaptation of \cite[Lemma~7.2]{rasmussen2007alternative} 
\begin{lemma}
\label{max/min of morse spectrum is determined}
Suppose that $\Phi$ fulfills the condition \eqref{bounded growth condition} and let $\{M_{1},...,M_{n}\}$ be the weak finest Morse decomposition for $\mathbb{P}\Phi$. Then for any $I\subset \{1,...,n\}:=J$ it holds true that 
$$\partial \Xi\left(\bigoplus_{i\in I}\mathbb{P}^{-1}M_{i}\right)(\omega)\subset \bigcup_{i\in I}\Xi(\mathbb{P}^{-1}M_{i})(\omega)$$
for any $\omega\in F$.
\end{lemma}

\begin{proof}
One has
\begin{align}
\label{Equation Maximima/Minima morse spectrum}
\Xi\left(\bigoplus_{i\in I}\mathbb{P}^{-1}M_{i}\right)(\omega)\supset \bigcup_{i\in I}\Xi(\mathbb{P}^{-1}M_{i})(\omega) \text{ for any } \omega\in F.
\end{align}
Now since for any $\omega\in F$  we have that $\Xi\left(\bigoplus_{i\in I}\mathbb{P}^{-1}M_{i}\right)(\omega)=[a,b]$ for some $a,b\in \mathbb{R}$ with $a\leq b$ (cf.~Proposition \ref{Proposition finitness} and Theorem \ref{basic props of morse}) it suffices to show that $a=\min \bigcup_{i\in I}\Xi(\mathbb{P}^{-1}M_{i}) $ and $b=\max \bigcup_{i\in I}\Xi(\mathbb{P}^{-1}M_{i}) $. For $\omega \in F$ and $i\in \{1,...,n\}$ let $P_{i}(\omega)$ be the projection with range $\mathbb{P}^{-1}M_{i}(\omega)$ and null-space $\bigoplus_{j\in J\setminus \{i\}}\mathbb{P}^{-1}M_{j}(\omega)$. Then, by Lemma \ref{Distance between attractor/repeller is tempered} there exists a tempered random variable $K:\Omega\to [1,\infty)$ such that $\|P_{i}(\omega)\|\leq K(\omega)$ for any $\omega \in F$. Let $\omega\in F$ and
assume that $b>\max \bigcup_{i\in I}\Xi(\mathbb{P}^{-1}M_{i}(\omega))$ (note that by \eqref{Equation Maximima/Minima morse spectrum} $b<\max \bigcup_{i\in I}\Xi(\mathbb{P}^{-1}M_{i}(\omega))$ is not possible) . Then there exists a sequence $\{(T_k, t_k, x_k)\}_{k \in \mathbb{N}}$  with  $T_k, t_k \in \mathbb{T}, |T_{k}|\geq |t_{k}|, x_k \in \bigoplus_{i\in I}\mathbb{P}^{-1} M_{i}(\theta_{t_k}\omega)\cap \mathbb{S}^{d-1}$ such that  $\lim_{k \to \infty} \lambda^{T_k}(\theta_{t_k}\omega, x_k) = b$. 
We obtain now for $k\in \mathbb{N}$:
\begin{align*}
    \lambda^{T_{k}}(\theta_{t_{k}}\omega,x_{k}) = & \frac{1}{T_{k}}\ln \left (\frac{\|\Phi(T_{k},\theta_{t_{k}}\omega)x_{k}\|}{\|x_{k}\|}\right)= \frac{1}{T_{k}}\ln \left (\frac{\|\sum_{i=1}^{n}\Phi(T_{k},\theta_{t_{k}}\omega) P_{i}(\theta_{t_{k}}\omega)x_{k}\|}{\|x_{k}\|}\right)  \\ &\leq  \frac{1}{T_{k}}\ln \left (\sum_{i=1}^{n}\frac{\|\Phi(T_{k},\theta_{t_{k}}\omega) P_{i}(\theta_{t_{k}}\omega)x_{k}\|}{\|x_{k}\|}\right) \\& \leq \frac{1}{T_{k}} \ln \left( \sum_{i=1,...,n,P_{i}(\theta_{t_{k}}\omega)x_{k}\neq 0}\frac{\|\Phi(T_{k},\theta_{t_{k}}\omega) P_{i}(\theta_{t_{k}}\omega)x_{k}\|}{\|P_{i}(\theta_{t_{k}}\omega)x_{k}\|/K(\theta_{t_{k}}\omega)} \right) 
     \\ & \leq \frac{1}{T_{k}} \ln \left( K(\theta_{t_{k}}\omega)n\max_{i=1,...,n,P_{i}(\theta_{t_{k}}\omega)x_{k}\neq 0}\frac{\|\Phi(T_{k},\theta_{t_{k}}\omega) P_{i}(\theta_{t_{k}}\omega)x_{k}\|}{\|P_{i}(\theta_{t_{k}}\omega)x_{k}\|}\right)\\
     &= \frac{1}{T_{k}}\ln(K(\theta_{t_{k}}\omega)n)+\max_{i=1,...,n,P_{i}(\theta_{t_{k}}\omega)x_{k}\neq 0}\lambda^{T_{k}}(\theta_{t_{k}}\omega,P_{i}(\theta_{t_{k}}\omega)x_{k})
\end{align*}
Using that $|T_{k}|\geq |t_{k}|$ for any $k\in \mathbb{N}$ and temperedness of $K(\cdot)$ we obtain a contradiction in the limit.
\\
We show now $a=\min \bigcup_{i\in I}\Xi(\mathbb{P}^{-1}M_{i}) $ and assume for contradiction  that $a<\min \bigcup_{i\in I}\Xi(\mathbb{P}^{-1}M_{i})$ (again $a>\min \bigcup_{i\in I}\Xi(\mathbb{P}^{-1}M_{i})$ is not possible because of \eqref{Equation Maximima/Minima morse spectrum}). Let $P(\omega)$ be the projection with range $\bigoplus_{i\in I}\mathbb{P}^{-1}M_{i}(\omega)$ and null-space $\bigoplus_{i\in J\setminus I}\mathbb{P}^{-1}M_{i}(\omega)$, i.e $P(\omega)=\sum_{i\in I}P_{i}(\omega)$. Choose an arbitrary $\varepsilon>0$ and let $\zeta:=\min \bigcup_{i\in I}\Xi(\mathbb{P}^{-1}M_{i})$. Define now for $i\in I$ the random-variables $H_{i}$ by $H_{i}(\omega)=1$ for $\omega\in \Omega\setminus F$
and for $\omega\in F$: $$H_{i}(\omega)=\inf \{K\geq 1: \forall t\geq 0 : \|\Phi(-t,\omega)P_{i}(\omega)\|\leq Ke^{-(\zeta-\varepsilon)t}\}.$$
Measurability of $H_{i}$ follows from completeness of $\mathcal{F}$, Lemma \ref{measurability projections} and continuity of $t\mapsto \|\Phi(-t,\omega)P_{i}(\omega)\|$. We show now that for arbitrary $i\in I$, $H_{i}$ is real valued and tempered. Suppose $H_{i}$ is not real-valued, then there exist sequences $(x_{k})_{k\in\mathbb{N}}\subset \mathbb{S}^{d-1}$ and $(t_{k})_{k\in\mathbb{N}}\subset \mathbb{T}_{+}$ such that for any $k\in \mathbb{N}$ one has $\|\Phi(-t_{k},\omega)P_{i}(\omega)x_{k}\|>ke^{-(\zeta-\varepsilon)t_{k}}$ this now implies that $\limsup_{k\to\infty}t_{k}=\infty$. By passing to a subsequence we may assume w.l.o.g that $\lim_{k\to\infty}t_{k}=\infty$. This gives us that 
$$\lambda^{t_{k}}(\theta_{-t_{k}}\omega,\Phi(-t_{k},\omega)P_{i}(\omega)x_{k})<\frac{1}{t_{k}}(\ln \|P_{i}(\omega)x_{k}\|-\ln(k)+(\zeta-\varepsilon)t_{k}) $$
as $\|P_{i}(\cdot)\|$ is tempered by Lemma~\ref{Distance between attractor/repeller is tempered} this implies that $\min \bigcup_{i\in I}\Xi(\mathbb{P}^{-1}M_{i})(\omega)\leq \zeta-\varepsilon$ which is a contradiction.\\
To the aim of a contradiction assume that $H_{i}$ is non-tempered then there exists a $\omega\in F$ such that $\limsup_{t\to\pm \infty}\frac{1}{|t|}\ln(H_{i}(\theta_{t}\omega))=\infty$. Consider a sequence $t_{n}\uparrow \infty$ such that $\limsup_{n\to \infty}\frac{1}{t_{n}}\ln(H(\theta_{t_{n}}\omega))=\infty$. Then for any $n\in\mathbb{N}$ there exists by definition of $H_{i}$ a $s_{n}\in\mathbb{T}_{+}$ such that for any $n\in\mathbb{N}$ $$\|\Phi(-s_{n},\theta_{t_{n}}\omega)P_{i}(\theta_{t_{n}}\omega)\|>\frac{1}{2}H_{i}(\theta_{t_{n}}\omega)e^{-(\zeta-\varepsilon) s_{n}}.$$
One obtains a contradiction to the growth condition and the fact that $\|P_{i}(\cdot)\|$ is tempered (cf. arguments in Lemma \ref{proof technical lemma finitness}).
\\
We have thus shown that for any $\varepsilon>0$ there exists a tempered random variable $\tilde{K}:\Omega\to [1,\infty)$ such that for any $\omega\in F$
$$\|\Phi(-t,\omega)P(\omega)\|\leq \tilde{K}(\omega)e^{-(\zeta-\varepsilon) t} \text{ for any } t\geq 0.$$ By applying similar arguments as in Lemma \ref{proof technical lemma finitness} it follows that for any $\omega\in F$,  $(\tilde{K}(\theta_{t_{n}}\omega))_{n\in\mathbb{N}}$ is bounded whenever $(t_{n})_{n\in\mathbb{N}}\subset \mathbb{T}$ is bounded. 
\\
Let now $\omega\in F$, $(T_{k})_{k\in\mathbb{N}}\subset\mathbb{T}$ be an arbitrary sequence with $T_{k}\uparrow\infty$, $(t_{k})_{k\in\mathbb{N}}\subset \mathbb{T}$, $|T_{k}|\geq |t_{k}|$ and $(x_{k})_{k\in\mathbb{N}}\subset \mathbb{S}^{d-1}$ with $x_{k}\in \mathbb{P}^{-1}M(\theta_{t_{k}}\omega)\setminus \{0\}$, then for any $k\in\mathbb{N}$ one has that (recall that $F$ is $\theta$-invariant.)
\begin{align*}
\|x_{k}\|= & \|\Phi(-T_{k},\theta_{t_{k}+T_{k}}\omega)\Phi(T_{k},\theta_{t_{k}}\omega)P(\theta_{t_{k}}\omega)x_{k}\| \\ &=\|\Phi(-T_{k},\theta_{t_{k}+T_{k}}\omega)P(\theta_{t_{k}+T_{k}}\omega)\Phi(T_{k},\theta_{t_{k}}\omega)x_{k}\| \\ &\leq  \tilde{K}(\theta_{T_{k}+t_{k}}\omega)e^{-(\zeta-\varepsilon)T_{k}}\|\Phi(T_{k},\theta_{t_{k}}\omega)x_{k}\|
\end{align*}
So $$\|\Phi(T_{k},\theta_{t_{k}}\omega)x_{k}\|\geq \frac{e^{(\zeta-\varepsilon) T_{k}}}{\tilde{K}(\theta_{t_{k}+T_{k}}\omega)}$$

This shows that $a\geq \zeta-\varepsilon$. Since $\varepsilon>0$ was arbitrary this finishes the proof. 
\end{proof}

We now prove the main statement, namely that $\Xi_{w}=\Sigma^{\prime}$ if \eqref{bounded growth condition} holds. For clarity we divide this proof into two propositions and one lemma. The proofs are an adaptation of \cite[Theorem~7.39]{rasmussen2007alternative}. The first proposition shows that one has $\Xi_{w} \subset \Sigma^{\prime}$.
\begin{proposition}
\label{morse spectrum contained in dichotomy spectrum}
    Suppose that \eqref{bounded growth condition} holds, then $\Xi_{w}\subset \Sigma^{\prime}$ almost surely.
\end{proposition}

\begin{proof}
Let $\{W_{1},...,W_{m}\}$ be the decomposition into the spectral manifolds of the dichotomy spectrum. Since by Theorem \ref{structure invariant manifolds} $\{\mathbb{P}W_{1},...,\mathbb{P}W_{m}\}$ defines in particular a weak Morse decomposition there exists for any $i\in \{1,...,n\}$ a $j\in \{1,...,m\}$ with $M_{i}\subset \mathbb{P}W_{j}$ almost surely. Let $\mu \in \Xi(\omega)$ for some $\omega\in F$ then, by Theorem \ref{basic props of morse} and $\theta$-invariance of $F$, for any $\omega\in F$ there exists $i\in \{1,...,n\}$ and sequences $T_{k}\uparrow \infty$, $(t_{k},x_{k})\in \mathbb{T}\times \mathbb{P}^{-1}M_{i}(\theta_{t_{k}}\omega)\setminus \{0\} $ with $|T_{k}|\geq |t_{k}|$ for $k\in\mathbb{N}$ such that $\lim_{k\to\infty} \lambda^{T_{k}}(\theta_{t_{k}}\omega,x_{k})=\mu$. Choose $j\in \{1,...,m\}$, such that $M_{i}\subset \mathbb{P}W_{j}$ almost surely. Note that since $M_{i}$ and $\mathbb{P}W_{j}$ are $\mathbb{P}\Phi$ invariant, the set $B\in \mathcal{F}$ of $\omega\in \Omega$ fulfilling  $M_{i}(\omega)\subset \mathbb{P}W_{j}(\omega)$ is $\theta$-invariant. Let $[a_{j},b_{j}]$ be the spectral interval of $W_{j}$. Assume that $\mu \not\in \Sigma^{\prime}$.  By assumption there exists an invariant projector $P_{\gamma_{j}},$ a tempered random variable $K:\Omega\to [1,\infty)$, an $\alpha>0$ and a $\theta$-invariant full-measure set $\tilde{F}\in \mathcal{F}$ such that for any $\omega\in \tilde{F}$:
 \begin{align}
 \label{eq:positive time}
     \|\Phi(t,\omega)P_{\gamma_{j}}(\omega)\|\leq K(\omega)e^{(\mu-\alpha)t} \text{ for } t\geq 0
 \end{align}
 and 
 \begin{align}
 \label{eq:negative time}
     \|\Phi(-t,\omega)(\mathbbm{1}-P_{\gamma_{j}}(\omega))\|\leq K(\omega)e^{-(\mu+\alpha)t} \text{ for } t\geq 0.
 \end{align}
 Moreover we may assume that for $\omega\in \tilde{F}$, $(K(\theta_{t_{n}}\omega))_{n\in \mathbb{N}}$ is bounded whenever $(t_{n})_{n\in\mathbb{N}}\subset \mathbb{T}$ is bounded.
 (compare arguments in Lemma \ref{proof technical lemma finitness})
 Let now $\hat{F}=F\cap \tilde{F}\cap B$ and $\omega \in \hat{F}$.
 \\
 Suppose first $\mu >b_{j}$.
 Because of $\mathbb{P}^{-1}M_{i}(\omega)\subset W_{j}(\omega)\subset \mathcal{R}(P_{\gamma_{j}}(\omega))$, \eqref{eq:positive time} implies that $\lim_{k\to\infty} \lambda^{T_{k}}(\theta_{t_{k}}\omega,x_{k})\leq\mu-\alpha$ which is a contradiction. Now suppose 
 that $\mu <a_{j}$. Then as  $\mathbb{P}^{-1}M_{i}(\omega)\subset W_{j}(\omega)\subset \mathcal{N}(P_{\gamma_{j-1}}(\omega))\subset \mathcal{N}(P_{\gamma_{j}}(\omega)) $ we obtain for any  $k\in\mathbb{N}$ (recall that $\hat{F}$ is $\theta$-invariant)
 \begin{align*}
\|x_{k}\| &=\|\Phi(-T_{k},\theta_{t_{k}+T_{k}}\omega)\Phi(T_{k},\theta_{t_{k}}\omega)(\mathbbm{1}-P_{\gamma_{j}}(\theta_{t_{k}}\omega))x_{k}\| \\ &= \|\Phi(-T_{k},\theta_{t_{k}+T_{k}}\omega) (\mathbbm{1}-P_{\gamma_{j}}(\theta_{t_{k}+T_{k}}\omega))\Phi(T_{k},\theta_{t_{k}}\omega)x_{k}\| \\ & \leq
K(\theta_{t_{k}+T_{k}}\omega)e^{-(\mu+\alpha)T_{k}}\|\Phi(T_{k},\theta_{t_{k}}\omega)x_{k}\|
 \end{align*}
 So for $k\in\mathbb{N}$:
 $$\|x_{k}\|K(\theta_{t_{k}+T_{k}}\omega)^{-1}e^{(\mu+\alpha)T_{k}}\leq \|\Phi(T_{k},\theta_{t_{k}}\omega)x_{k}\|$$
 which  implies $\lim_{k\to\infty} \lambda^{T_{k}}(\theta_{t_{k}}\omega,x_{k})\geq \mu+\alpha$, yielding a contradiction. 
\end{proof}
 
\begin{lemma}
\label{max/min of morse spectrum is given by dichotomy}
Suppose that $\Phi$ fulfills the condition \eqref{bounded growth condition}. Then  $\Xi(\mathbb{R}^{d})=[\min \Sigma^{\prime}, \max \Sigma^{\prime}]$ almost surely. In particular $\min \Xi_{w}=\min \Sigma^{\prime}$ and $\max \Xi_{w}=\max \Sigma^{\prime}$ almost surely. 
\end{lemma}

\begin{proof}
 It will only be shown that $\max \Xi(\mathbb{R}^{d})=\max \Sigma^{\prime}$ since $\min \Xi(\mathbb{R}^{d})=\min \Sigma^{\prime}$ can be proven similarly. Assume to the aim of a contradiction that $\max \Xi(\mathbb{R}^{d})< \max \Sigma^{\prime}$ almost surely (note that $\max \Xi(\mathbb{R}^{d})> \max \Sigma^{\prime}$ is not possible due to Proposition \ref{morse spectrum contained in dichotomy spectrum} and Lemma \ref{max/min of morse spectrum is determined}).
Then in particular there exists an $\alpha>0$ such that the $\theta$-invariant set $\tilde{F}\in \mathcal{F}$ where \\ $\limsup_{t\to\infty} \frac{1}{t} \ln \|\Phi(t,\omega)\|< \max \Sigma^{\prime}-\alpha$ has full-measure. Let $\hat{F}:=F\cap \tilde{F}$. Finally, define 
$$ H(\omega )= \begin{cases} \inf \{K\geq 1: \forall t\geq0: \ \|\Phi(t,\omega)\|\leq Ke^{(\max \Sigma^{\prime}-\alpha) t}\}  & \text{if } \omega \in \hat{F} \\ 1 \text{ else. } \end{cases} $$
Note that $H$ is measurable as $\mathcal{F}$ is complete and $\mathbb{T}\to \mathbb{R}$, $t\mapsto \|\Phi(t,\omega)\|$ is continuous. 
\\
From $\limsup_{t\to\infty} \frac{1}{t} \ln \|\Phi(t,\omega)\|< \max \Sigma^{\prime}-\alpha$ for all $\omega\in \hat{F}$ and \eqref{bounded growth condition} we obtain that $H$ is real-valued and that $H$ is tempered (cf. arguments made in Lemma \ref{max/min of morse spectrum is determined}). This, however, implies a contradiction.
\end{proof}

\begin{proposition}
\label{dichotomy is contained in Morse}
Suppose that $\Phi$ fulfills the condition \eqref{bounded growth condition}. Then it holds that $\Xi_{w}\supset\Sigma^{\prime}$ almost surely. 
\end{proposition}

\begin{proof}
 If $\Xi_{w}(\omega)=[\min \Sigma^{\prime}, \max \Sigma^{\prime}]$ for every $\omega\in F$, there is nothing to show. Suppose this is not the case and let $\mu\not \in \Xi_{w}\cap [\min \Sigma^{\prime}, \max \Sigma^{\prime}] $. Consider the unique finest weak Morse decomposition $\{M_{1},...,M_{n}\}$ and let $[a_{i},b_{i}]=\Xi(\mathbb{P}^{-1}M_{i})(\omega)$  for $\omega\in F$, $i\in \{1,...,n\}$ with $b_{0}=-\infty$ and $a_{n+1}=\infty$. By relabeling the Morse sets we may choose a $j\in \{2,...,n\}$ so that $\mu>b_{j-1}$, $\mu<a_{j}$  and $\max \bigcup_{i=1}^{j-1}[a_{i},b_{i}]<\mu <\min \bigcup_{i=j}^{n}[a_{i},b_{i}]$. We define a projector $P(\omega)$ with range $\mathbb{P}^{-1}M_{1}(\omega)\oplus...\oplus \mathbb{P}^{-1}M_{j-1}(\omega)$ and null space $\mathbb{P}^{-1}M_{j}(\omega)\oplus...\oplus \mathbb{P}^{-1}M_{n}(\omega)$ (note that this is measurable by Lemma \ref{measurability projections}). Moreover by Lemma \ref{Distance between attractor/repeller is tempered} it holds that $\|P(\cdot)\|\leq \tilde{K}(\cdot)$ for some tempered random variable $\tilde{K}:\Omega \to [1,\infty)$. Assume w.l.o.g that $\lim_{t\to\pm \infty}\frac{1}{|t|}\ln K(\theta_{t}\omega)=0$ for all $\omega\in F$.
\\
Consider for fixed  $\delta>b_{j-1}$ and $\eta<a_{j}$ the random variable $H:\Omega\to \overline{\mathbb{R}}$ defined by $H(\omega)=1$ for $\omega\in \Omega\setminus F$ and for $\omega\in F$:
$$ H(\omega )=  \inf \{K\geq 1: \forall t\geq0: \ \|\Phi(t,\omega)P(\omega)\|\leq Ke^{\delta t} \text{ and } \|\Phi(-t,\omega)(\mathbbm{1}-P(\omega))\|\leq Ke^{-\eta t}\}. $$
One can see that $H$ is measurable since $\mathcal{F}$ is complete, $\mathbb{T}\to \mathbb{R}$, $t\mapsto \|\Phi(t,\omega)P(\omega)\|$ is continuous, and $\omega\mapsto \|\Phi(t,\omega)P(\omega)\|$ is measurable (compare Lemma \ref{measurability projections}) .
\\
It suffices to show that $H$ is tempered and real-valued to show that $\mu\not\in \Sigma^{\prime}$. This is done in two steps:
\\
\emph{Step 1: $H$ is real-valued:}
\\
We argue by contradiction. Assume first  there exists a $\omega\in F$ and sequences $(x_{n})_{n\in \mathbb{N}}\subset \mathbb{S}^{d-1}$, $(t_{n})_{n\in \mathbb{N}}\subset \mathbb{R}_{+}$ such that for any $n\in \mathbb{N}$
$$\|\Phi(t_{n},\omega)P(\omega)x_{n}\|>ne^{\delta t_{n}}$$
this implies that $\limsup_{n\to\infty}t_{n}=\infty$.
We obtain $$\lambda^{t_{n}}(\omega,P(\omega)x_{n})=\frac{1}{t_{n}}\ln \frac{\|\Phi(t_{n},\omega)P(\omega)x_{n}\|}{\|P(\omega)x_{n}\|}\geq \frac{1}{t_{n}}(\ln(n)+\delta t_{n}-\ln(\tilde{K}(\omega)))$$
but this implies that $\limsup_{n\to\infty}\lambda^{t_{n}}(\omega,P(\omega)x_{n})\geq \delta>b_{j-1}$,
which in turn implies that $\sup \Xi(\mathbb{P}^{-1}M_{1}\oplus...\oplus \mathbb{P}^{-1}M_{j-1})\geq \delta >b_{j-1}$  which is a contradiction to Lemma \ref{max/min of morse spectrum is determined}.
\\
Assume similarly that there exists a sequence $(x_{n})_{n\in\mathbb{N}}\subset \mathbb{S}^{d-1}$ and $(t_{n})_{n\in\mathbb{N}}\subset \mathbb{R}_{+}$ such that for any $n\in\mathbb{N}$
$$\|\Phi(-t_{n},\omega)(\mathbbm{1}-P(\omega))x_{n}\|>ne^{-\eta t_{n}}.$$
Again it follows that $\limsup_{n\to\infty}t_{n}=\infty$. Letting $y_{n}=\Phi(-t_{n},\omega)(\mathbbm{1}-P(\omega))x_{n}\in \bigoplus_{i=j}^{n}\mathbb{P}^{-1}M_{i}\setminus \{0\}$ 
we get 
$$\lambda^{t_{n}}(\theta_{-t_{n}}\omega,y_{n})=\frac{1}{t_{n}}\ln \frac{\|(\mathbbm{1}-P(\omega))x_{n}\|}{\|y_{n}\|}\leq \frac{1}{t_{n}}(-\ln(n)+t_{n}\delta+\ln(\tilde{K}(\omega)))$$
 so $\liminf_{n\to\infty} \lambda^{t_{n}}(\theta_{-t_{n}}\omega,y_{n})\leq \eta<a_{j}$.
This implies $\inf \Xi(\mathbb{P}^{-1}M_{j}\oplus...\oplus \mathbb{P}^{-1}M_{n})\leq \eta <b_{j-1}$ which is again a contradiction to Lemma \ref{max/min of morse spectrum is determined}.
\\
\emph{Step 2: $H$ is tempered:}
\\
To the aim of a contradiction assume that $H$ is non-tempered then there exists a $\omega\in F$ such that $\limsup_{t\to\pm \infty}\frac{1}{|t|}\ln(H(\theta_{t}\omega))=\infty$. Now consider a sequence $t_{n}\uparrow \infty$ such that $\limsup_{n\to \infty}\frac{1}{t_{n}}\ln(H(\theta_{t_{n}}\omega))=\infty$ . Then for any $n\in\mathbb{N}$ there exists by definition of $H$ a $s_{n}\in\mathbb{T}_{+}$ such that for any $n\in\mathbb{N}$ 
\begin{align}
\label{eq: temperedness Prop Morse contains dichotomy pos time}
\|\Phi(s_{n},\theta_{t_{n}}\omega)P(\theta_{t_{n}}\omega)\|>\frac{1}{2}H(\theta_{t_{n}}\omega)e^{\delta s_{n}}
\end{align}
or 
\begin{align}
\label{eq: temperedness Prop Morse contains dichotomy neg time}
\|\Phi(-s_{n},\theta_{t_{n}}\omega)(\mathbbm{1}-P(\theta_{t_{n}}\omega))\|>\frac{1}{2}H(\theta_{t_{n}}\omega)e^{-\eta s_{n}}.
\end{align}
We only consider the case that \eqref{eq: temperedness Prop Morse contains dichotomy pos time} holds for infinitely many $n\in\mathbb{N}$ as the case that \eqref{eq: temperedness Prop Morse contains dichotomy pos time} holds for infinitely many $n\in\mathbb{N}$ can be treated analogously. Suppose w.l.o.g that \eqref{eq: temperedness Prop Morse contains dichotomy pos time} holds for any $n\in\mathbb{N}$. If $|s_{n}|\geq |t_{n}|$ for infinitely many $n\in\mathbb{N}$ we obtain that  $\sup \Xi(\mathbb{P}^{-1}M_{1}\oplus...\oplus \mathbb{P}^{-1}M_{j-1})\geq \delta >b_{j-1}$  which is a contradiction to Lemma \ref{max/min of morse spectrum is determined}.  Now if $|s_{n}|\leq |t_{n}|$ for w.l.o.g every $n\in\mathbb{N}$ we have in particular by the bounded growth condition \eqref{bounded growth condition} that $\tilde{K}(\theta_{t_{n}}\omega)e^{a|s_{n}|}>\frac{1}{2}H(\theta_{t_{n}}\omega)e^{\delta s_{n}}$. This leads to the contradiction $a>\infty$. 
\end{proof}

\section{Discussion}
In this paper, we have only considered the Morse spectrum induced by a unique finest weak Morse decomposition. However, given a non-trivial attractor $A$ pertaining to some stronger notion of attraction/repulsion, such that the corresponding weak repeller $R$ also pertains to this notion of attraction/repulsion, Theorem \ref{Theorem unique finest Morse decomposition} implies that one obtains a unique finest Morse decomposition pertaining to this notion of attraction/repulsion. More concretely one may consider the following notions of attraction/repulsion:
\begin{itemize}
    \item[(i)] Pullback (pb) attractor-repeller pairs: We call a weak attractor $A$ a (local) pullback attractor if some open neighborhood $U$ of $A$ is pullback attracted to $A$. It can be shown that the corresponding weak repeller $R$ will also pullback repel an open neighborhood of $R$ (cf.~\cite[Lemma 5.1]{Liu2007}).
    \item[(ii)] Uniform weak (uw) attractor-repeller pairs: We call a weak attractor $A$ a (local) uniform weak attractor if the associated basin of attraction essentially contains a ball around $A$ of deterministic radius. It is clear that the basin of repulsion associated to the corresponding weak repeller $R$ will also essentially contain a ball of deterministic radius around $R$.
    \item[(iii)] Uniform strong (us) attractor-repeller pairs: We call a weak attractor $A$ a (local) uniform strong attractor if $A$ both forward and pullback attracts a ball of deterministic radius around $A$. It can be shown that the corresponding weak repeller $R$ both pullback and forward repels a ball of deterministic radius around $R$ (cf.~ \cite[Theorem~3.3.5]{Call2014}).
\end{itemize}
Using Theorem \ref{Theorem unique finest Morse decomposition} one obtains unique finest Morse decompositions constructed from attractor-repeller pairs pertaining to the above notions of attraction/repulsion. This in turn allows us to define pullback (pb), uniform weak (uw) and uniform strong (us) Morse spectra, which we denote by $\Xi_{j}$, $j\in \{w,pb,uw,us\}$. It follows easily that for $j\in \{uw,us\}$ the angles between the subspaces associated to the Morse sets will be essentially bounded away from each other. Thus, one may consider $\Xi_{uw}$ and $\Xi_{us}$ as uniform Morse spectra. Similarly there exists a uniform notion of the dichotomy spectrum: If one demands that the random variable $K$ in \eqref{def exp dichotomy neg time} and \eqref{def exp dichotomy pos time} is essentially bounded, then one easily obtains that the angles between the range and nullspace of the corresponding invariant projector are essentially bounded away from zero. Let $\Sigma$ denote the uniform dichotomy spectrum.

In Theorem \ref{structure invariant manifolds} we have stated that if one considers an invariant projector $P$ such that $\Phi$ admits an exponential dichotomy with $P$ and tempered random variable $K$, then one has that the null-space of $P$ is a weak attractor with corresponding repeller given by the range of $P$. In fact a stronger statement holds, namely 
\begin{theorem}
    Suppose $P$ is an invariant projector such that \eqref{def exp dichotomy pos time} and \eqref{def exp dichotomy neg time} holds for some $\gamma\in \mathbb{R}$, $\alpha>0$ and some tempered random variable $K$. Then $(A,R):=(\mathbb{P}\mathcal{N}(P),\mathbb{P}\mathcal{R}(P))$ defines a pullback attractor-repeller pair. If one additionally assumes that $K$ is essentially bounded, then $(A,R)$ is a uniform strong attractor-repeller pair.
\end{theorem}
Given this result, the proof of Theorem \ref{morse spectrum equals dichotomy spectrum} implies that also $\Xi_{pb}=\Sigma^{\prime}$ and hence $\Xi_{w}=\Xi_{pb}$ under the assumption that \eqref{bounded growth condition} holds. Similarly one obtains that $\Xi_{uw}=\Xi_{us}$ if \eqref{bounded growth condition} holds. Additionally one has that $\Xi_{uw}=\Xi_{us}=\Sigma$ whenever $K$ in \eqref{bounded growth condition} is essentially bounded.
For more details see \cite{alqaiwani2024spectra}. However, one generally has that $\Xi_{w}\subsetneq \Xi_{uw}$ as the following example demonstrates. 
\begin{example}
\label{attr and rep coming arbitrarily close}
    Consider $(\Omega, \mathcal{F}, \mu) = ([0,1), \mathcal{B}([0,1)), \lambda|_{[0,1)})$ and 
$\theta: \Omega \to \Omega$,  $\theta(\omega) = \omega + \alpha \mod 1$ with $\alpha \in \mathbb{R}\setminus\mathbb{Q}$
and the random dynamical system generated via 
\begin{align*}
\Phi (1,\omega) = 
\begin{pmatrix}
\beta & 0 \\
0 & 1
\end{pmatrix}, \beta > 1.\\
\end{align*}
$\mathbb{P}\Phi$ has the uniform strong attractor-repeller pair $(\mathbb{P}(1,0), \mathbb{P}(0,1)) = (A,R)$. This implies that one has 
$\Xi_{\Phi,j}=\{0,\ln(\beta)\}$ for $j\in \{w,pb,uw,us\}$.
\\
Now let $$H(\omega) = 
\begin{pmatrix}
1 & e^{\omega} \\
1 & e^{1 - \omega}
\end{pmatrix} \omega \neq \frac{1}{2}, \text{ and }  H\left( \frac{1}{2} \right) = 
\begin{pmatrix}
1 & 0 \\
0 & 1 
\end{pmatrix}.$$
and define $$\psi (1,\omega) = H(\theta\omega) \Phi (1,\omega) H^{-1} (\omega).$$
Then $\mathbb{P} \psi$ has the pullback attractor-repeller pair \\ $(\tilde{A}(\omega),\tilde{R}(\omega))= (\mathbb{P} (1,1), \mathbb{P} (e^{\omega},e^{1-\omega)})$. So
$\tilde{A} (\theta_{n}\omega) = \mathbb{P}(1,1), \tilde{R} (\theta_{n}\omega) = \mathbb{P} (\theta_{n}\omega,1-\theta_{n}\omega) $.
\\
However, the only uniform weak attractor-repeller pair for $\mathbb{P}\psi$ is $\{\mathbb{P}^{1},\emptyset \}$
as for any $\omega \in [0,1) \setminus \{\frac{1}{2}\} $ there is a sequence $(n_{k})_{k\in\mathbb{N}}\subset \mathbb{N}$ such that $\lim_{k\to\infty}\theta_{n_{k}}\omega = \frac{1}{2}$. As $H$ is a tempered random coordinate change one has $\Xi_{\psi,j}=\{0,\ln(\beta)\}$ (cf. Proposition \ref{Prop morse spectrum coordinate change}). However, one sees that $\Xi_{\psi,uw}=\Xi_{\psi,us}=[0,\ln(\beta)]$.
\end{example}

\bibliographystyle{plain}
\bibliography{literatur}
\end{document}